\documentclass{amsart}
\usepackage{amsmath, amsthm, amssymb, amsfonts}
\usepackage{graphicx, tikz, units, url}
\usepackage{color}

\usepackage{mathrsfs} 
\usepackage{bm,bbm} 
\usepackage{colonequals} 
\usepackage{physics} 
\usepackage{extarrows} 
\usepackage{mathtools} 

\usepackage{enumitem}
\usepackage{scalerel} 
\usepackage{stackengine}

\usepackage[colorlinks,citecolor=blue,linkcolor=blue,backref=page]{hyperref}
\usepackage[nameinlink]{cleveref}

\newcommand{\sA}{\mathcal{A}}

\newcommand{\sN}{\mathcal{N}}

\newcommand{\sR}{\mathcal{R}}
\newcommand{\sQ}{\mathcal{Q}}
\newcommand{\sV}{\mathcal{V}}

\newcommand{\N}{\mathbb{N}}
\newcommand{\Z}{\mathbb{Z}}

\newcommand{\R}{\mathbb{R}}

\newcommand{\D}{\mathbb{D}}

\DeclareMathOperator{\diam}{Diam}

\theoremstyle{plain}
\newtheorem{ENPdefinition}{Definition}

\newtheorem{theorem}{Theorem}[section] 
\newtheorem{lemma}[theorem]{Lemma} 
\newtheorem{proposition}[theorem]{Proposition} 
\newtheorem{corollary}[theorem]{Corollary}

\newtheorem*{construction*}{Construction}
\newtheorem*{theorem*}{\bf Theorem}
\newtheorem*{corollary*}{Corollary}

\newtheorem*{observation*}{Observation}
\newtheorem*{claim*}{Claim}

\theoremstyle{plain}

\newtheorem*{remark*}{Remark}
\newtheorem*{question*}{Question}

\theoremstyle{definition}
\newtheorem{definition}[theorem]{Definition} 
\newtheorem*{def*}{Definition}
\allowdisplaybreaks
\numberwithin{equation}{section}
\begin{document}

\title[Eigenvalue estimates and P\'olya's conjecture]
{New Qualitative and Quantitative Estimates for  Neumann Eigenvalues toward P\'olya's Conjecture}

\author{Xiang He, Renjin Jiang, Fanghua Lin and Zuoqin Wang}
	\thanks{X. He was partially supported by the China Postdoctoral Science Foundation No. 2024M761591 \& NNSFC No. 12601107; R. Jiang was partially supported by NNSFC No. 12526205 \& No. 12471094; F. Lin was
partially supported by National Science Foundation grant DMS2247773.; Z. Wang was partially supported by NNSFC No. 12571064. }
\address{Yau Mathematical Sciences Center\\
	Tsinghua University\\
	Beijing, 100084\\ P.R. China\\}
\email{x-he@mail.tsinghua.edu.cn}

\address{Academy for Multidisciplinary Studies\\
	Capital Normal University\\
	Beijing, 100084\\ P.R. China\\}
\email{rejiang@cnu.edu.cn}

\address{Courant Institute of Mathematical Sciences\\
251 Mercer Street, New York, NY 10012, USA}
\email{linf@cims.nyu.edu}

\address{School of Mathematical Sciences\\
	University of Science and Technology of China\\
	Hefei, 230026\\ P.R. China\\}
\email{wangzuoq@ustc.edu.cn}

\begin{abstract}
	We introduce the Enhanced Neumann P\'olya inequality, which is stronger than P\'olya’s conjecture, and verify it for many classes of domains, including strip-tiling domains, Euclidean balls, thin product domains, and certain higher-dimensional tiling domains. We develop a geometric removal approach for the Enhanced Neumann P\'olya inequality.
	In particular, our results provide new classes of domains satisfying the Neumann part of P\'olya's conjecture.
\end{abstract}

\maketitle

\section{Introduction}

\subsection{Geometric removal method  toward P\'olya's Conjecture}
Let $\Omega$ be a bounded domain in $\R^n$. The Dirichlet Laplacian on $\Omega$ has a discrete spectrum
\[
0< \lambda_1(\Omega)< \lambda_2(\Omega)\leq \lambda_3(\Omega) \leq \cdots \nearrow +\infty.
\]
Assume throughout that $\Omega$ has Lipschitz boundary, or that  the embedding $H^1(\Omega) \subset L^2(\Omega)$ is compact. Then the Neumann Laplacian on $\Omega$ also has a discrete spectrum,
\[
0= \mu_0(\Omega)< \mu_1(\Omega)\leq \mu_2(\Omega) \leq \cdots \nearrow +\infty.
\]	
Denote the Dirichlet eigenvalue counting function by
\[
\sN^D_\Omega(\lambda):=\#\{k:\ \lambda_k(\Omega)< \lambda\},
\]
and the Neumann eigenvalue counting function by
\[
\sN^N_\Omega(\lambda):=\#\{k:\ \mu_k(\Omega)< \lambda\}.
\]

In 1954, P\'olya \cite{Pol54} conjectured that for any $k\ge 1$,
\[
\lambda_k(\Omega)\ge \frac{4\pi^2}{(\omega_n |\Omega|)^{2/n}}k^{\frac 2n}\ge\mu_{k}(\Omega).
\]
To reformulate this conjecture in terms of the  eigenvalue counting functions, we first introduce the general semiclassical constants. For any $\gamma\ge0$ and $n\ge0$, let
\begin{equation}\label{Lgammad}
	L_{\gamma, n}= \frac{\Gamma(\gamma+1)}{(4\pi)^{n/2} \Gamma(\gamma+1+n/2)}.
\end{equation}
It is straightforward to verify that these constants satisfy the product formula
\begin{equation}\label{Cd1d2}
	L_{0,m} L_{\frac{m}{2},n} = L_{0,m+n}, \qquad \forall  m,n \ge 0.
\end{equation}
P\'olya's conjecture can be equivalently formulated as
\begin{equation}\label{PolyaConj}
\sN^D_{\Omega}(\lambda)\le L_{0,n}|\Omega|\lambda^{\frac n2}\le \sN^N_{\Omega}(\lambda), \qquad \forall \lambda>0.
\end{equation}

In this paper, we focus on the Neumann part of P\'olya's conjecture. Since every bounded domain in $\R$ satisfies P\'olya's conjecture,  we assume throughout that ${n\ge 2}$.
  
To state our main results, we first introduce the \ref{ENP} inequality, which provides an abstract framework for the Neumann part of P\'olya's conjecture and serves as a unified quantitative condition for  establishing P\'olya's conjecture for Neumann eigenvalues via geometric removal.
\begin{ENPdefinition}\label{ENP}
	We say that a bounded domain $\Omega \subset \mathbb R^n$ satisfies the
	\emph{Enhanced Neumann P\'olya inequality} (the \emph{\ref{ENP} inequality} for short)
	if there exist positive constants $C_1(\Omega)$ and
	$C_2(\Omega)$ such that 
	\[
	\sN_\Omega^N(\lambda) \ge
	L_{0,n}|\Omega|\lambda^{\frac n2}
	+ 	\max\left( 	C_1(\Omega)\lambda^{\frac{n-1}2}, C_2(\Omega) \right), \qquad \forall \lambda>0.
	\]
\end{ENPdefinition}

Although the \ref{ENP} inequality is strictly stronger than P\'olya's conjecture, it is satisfied by a broad class of domains. In particular,as shown in Sections~\ref{sec3} and~\ref{app}, it holds for many domains for which P\'olya's conjecture has already been established.

Our first result establishes the stability of the \ref{ENP} inequality under a geometric removal procedure. As a consequence, this method yields a large class of Euclidean domains for which the Neumann part of P\'olya's conjecture holds. This result may be viewed as a natural but nontrivial extension of the recent progress on the Dirichlet case obtained in \cite{JL25}. The corresponding result in \cite[Theorem 5.2]{HW24} requires an additional geometric restriction on the location of the obstacle, whereas our current result removes this restriction entirely.

\begin{theorem}\label{general}
	Let $\Omega\subset\mathbb{R}^n$ be a bounded domain satisfying the \ref{ENP} inequality, and let $U\subset\mathbb{R}^n$ be a bounded domain with smooth boundary. Then there exists a constant $a_0>0$, depending only on $\Omega$ and $U$, such that, for every $a\in(0,a_0)$ and every $x\in\mathbb{R}^n$ satisfying $\overline{x+aU}\subset\Omega$,  
	the domain 
	\[
	\Omega'=\Omega\setminus\overline{x+aU}
	\]
	also satisfies the \ref{ENP} inequality. In particular, the Neumann eigenvalues of $\Omega'$ satisfy P\'olya's conjecture.
\end{theorem}

Although the theorem is stated for a connected domain $U$ with smooth boundary, the connectedness assumption is inessential. Indeed, by applying the theorem iteratively, one may successively remove finitely many sufficiently small holes, regardless of their shapes or locations.  The smoothness assumption on $U$ is imposed only to ensure the applicability of Seeley's asymptotic formula (cf. Lemma \ref{Seeasy}) for the eigenvalue counting function. Consequently, it may be replaced by any regularity condition under which the required asymptotic estimate remains valid. In particular, the same argument applies to polygons and, more generally, to any class of domains for which the corresponding eigenvalue-counting asymptotics are available. We also point out that the assumption   $\overline{x+aU} \subset \Omega$ may be weakened to ${x+aU}\subset\Omega$, thus allowing the boundary of the removed hole to touch $\partial\Omega$,  provided that the resulting domain $\Omega'=\Omega\setminus\overline{x+aU}$ has purely discrete Neumann spectrum.

\begin{remark*}
    Theorem \ref{general} admits a Dirichlet analogue. More precisely, we say that a bounded domain $\Omega$ satisfies the \emph{EDP inequality} if  there exists a constant $C(\Omega)>0$ such that 
	\begin{equation}\label{AssDiri}
	\sN_\Omega^D(\lambda) \le
	L_{0,n}|\Omega|\lambda^{\frac n2}
	\left(1-C(\Omega)\lambda^{-\frac12}\right)_+, \qquad  \forall \lambda>0.
	\end{equation}
	By an argument essentially identical to that used in \cite[Theorem 1.9]{JL25}, one can show that the EDP inequality is also stable under the removal of sufficiently small holes. More precisely, for any $U$, there exists a constant $a_0>0$, depending only on $\Omega$ and $U$, such that, for every $a\in(0,a_0)$ and every $x\in\mathbb R^n$ satisfying	$x+aU\subset\Omega$, 
	the domain 	$\Omega\setminus\overline{(x+aU)}	$  also satisfies the EDP inequality.  We omit the details and focus only on the Neumann case in this paper.
\end{remark*}
 
A number of stability results for Neumann eigenvalues under the removal of a small interior hole have already been established in the literature; see, for example, \cite{RT75,LNS11,Kab20,LM12,FLO25}. However, the available results typically impose geometric restrictions on the location of the hole and therefore do not apply to the setting of Theorem~\ref{general}, where the hole may be placed arbitrarily inside $\Omega$. To overcome this difficulty, we establish in Lemma~\ref{genhssta} a tailored one-sided comparison between the Neumann eigenvalues of $\Omega$ and $\Omega'$, which serves as the key ingredient in the proof.

Theorem \ref{general} is essentially qualitative, since our method does not yield explicit geometric constants (cf. the formula \eqref{eq:theexpressionofa}). The next theorem studies a special case in which the removed subsets are cubes, and gives a quantitative refinement in terms of their total surface area and volume. Although the thresholds  $A(M,\Omega)$ and $V(M,\Omega)$ remain implicit for a general domain, they can be computed explicitly via geometric data when $\Omega$ is a cuboid; see  Theorem \ref{thm:cuboid}.

Throughout the paper, for any cube $Q$, we denote by
\[
\ell(Q):=\mbox{the side length of a cube} \ Q.
\]

\begin{theorem}\label{mthm}
    Let $\Omega\subset \R^n$ be a bounded domain satisfying the \ref{ENP} inequality. For any $M\in\Z_+$, let $\{Q_i\}_{i=1}^M$ be a family of pairwise disjoint open cubes whose closures are contained in $\Omega$. Then there exist positive constants $A(M,\Omega)$ and $V(M,\Omega)$ such that, if
    \[
    \sA_{\sQ}:=\sum_{k=1}^M \ell(Q_k)^{n-1}<A(M,\Omega),\qquad
    \sV_{\sQ}:=\sum_{k=1}^M \ell(Q_k)^n<V(M,\Omega),
    \]
    the Neumann Laplacian eigenvalues of
    \[\Omega'=\Omega\setminus \overline{\cup_{k=1}^M Q_k}\]
    satisfy P\'olya's conjecture.
\end{theorem}
 
The corresponding Dirichlet result has already been proved in \cite{JL25}. More precisely,
given any $\Omega$ satisfying the EDP inequality \eqref{AssDiri}, if $\{Q_i\}_{i\in\mathbb N}$ is an {admissible} class of cubes contained in $\Omega$ and
	\[
	\sA_{\sQ}:=\sum_{i\in\N}\ell(Q_i)^{n-1}\le
	\frac{|\Omega|C(\Omega)}
	{n^{\frac32}\pi},
	\]
then the Dirichlet eigenvalues of the Laplacian on $\Omega'=\Omega\setminus \overline{\cup_{i\in \N}Q_i}$ 
satisfy P\'olya's conjecture. 
			 
The Neumann case is more delicate because finitely many low-lying eigenvalues require additional treatment. Consequently, Theorem \ref{mthm} assumes that the removed cubes form a finite family and that their total volume is sufficiently small. The finiteness of the family is essential in our argument: since $\sN_{Q_i}^N(\lambda)\ge1$ for every $i$ and $\lambda>0$, infinitely many removed cubes would lead to uncontrolled accumulation.

\subsection{Domains satisfying the \ref{ENP} inequality with explicit constants}

We next compute the constants $C_1(\Omega)$, $C_2(\Omega)$ in the \ref{ENP} inequality explicitly for strip-tiling domains and planar disks, and thus produce many domains satisfying Neumann P\'olya's conjecture. Let us recall the strip-tiling domains introduced in \cite{JL25}.

\begin{definition}\cite[Strip-tiling Domain]{JL25}
	A bounded domain $\Omega\subset\R^n$ is called a \emph{strip-tiling domain} if there exists a constant $w>0$ such that, by translations and rotations in the first $n-1$ directions,
$\Omega$ tiles $\R^{n-1}\times (0,w)$. 
\end{definition}

The value of the width $w$ is inessential, since under a scaling of the domain, the \ref{ENP} inequality is preserved with scaled constants. 

\begin{proposition}\label{strip}
	Every strip-tiling domain satisfies the \ref{ENP} inequality. Moreover, for a strip-tiling domain with  $w=\pi$,  the constants in the \ref{ENP} inequality can be chosen as 
	\[
	C_1(\Omega)=\frac{\left(4-\pi\right)L_{0,n-1}}{4\pi}|\Omega|, \qquad 
	C_2(\Omega) = 	\min\left( 	\frac12, \frac{\left(4-\pi\right)|\Omega|^{\frac1n}}{2\sqrt\pi}
	L_{0,n-1} 	\right).
	\]
\end{proposition}

Substituting the constant $C_1(\Omega)$
from Proposition~\ref{strip} into \eqref{1} and \eqref{2}, we obtain the following corollary.
\begin{corollary}\label{cor:strip}
	Let $\Omega$ be a strip-tiling domain tiling
	$\R^{n-1}\times(0,\pi)$.
	For any $M\in\Z_+$, let
	$\{Q_i\}_{i=1}^M$
	be a family of pairwise disjoint open cubes whose closures are contained in $\Omega$, 
	such that 
	\[
	\sA_{\sQ} \le \frac{(4-\pi)|\Omega|\pi^{n-2}L_{0,n-1}}{16n}.
	\] 
	Then every Neumann eigenvalue of $\Omega'=\Omega\setminus \overline{\cup_{k=1}^M Q_k}$ with 
	\[
	\mu_k(\Omega') >  \left( \frac{8\pi M}{(4-\pi)|\Omega|L_{0,n-1}} \right)^{\frac{2}{n-1}}
	\]
	satisfies P\'olya's inequality.
\end{corollary}

\begin{remark*}
	The lower bound on $\mu_k(\Omega')$ is needed here because the proof of Theorem~\ref{mthm} relies on an explicit $L^\infty$ estimate for Neumann eigenfunctions (see \eqref{Linftybd}). While such an estimate is available in general, no explicit upper bound for the corresponding constant is currently known for strip-tiling domains. Thus Corollary \ref{cor:strip} reduces P\'olya's conjecture to verifying P\'olya's inequality for the finitely many eigenvalues below the explicit threshold above.
\end{remark*}

	Next we turn to Euclidean balls. The works \cite{FLPS23,li2026p,FLPS26disk} together establish P\'olya's conjecture for Euclidean balls in all dimensions. In particular, for the Neumann case the strict inequality 
	\[
	\sN_B^N(\lambda) > L_{0,n}|B|\lambda^{\frac n2}, \qquad \forall \lambda>0
	\]
	was established for every Euclidean ball $B\subset\R^n$. 
	Moreover, the Neumann counting function of a Euclidean ball satisfies
	the two-term Weyl asymptotic
	\[
	\sN_B^N(\lambda) = L_{0,n}|B|\lambda^{\frac n2}
	+ \frac14 L_{0,n-1}|\partial B| \lambda^{\frac{n-1}{2}} + o\left(\lambda^{\frac{n-1}{2}}\right) \qquad\text{as }\lambda\to\infty;
	\]
	see \cite{Ivrii80,SafarovVassiliev97}. Combining the strict P\'olya's inequality with the two-term Weyl asymptotics, we
	conclude that every Euclidean ball satisfies the \ref{ENP} inequality.
	
	Theorem~\ref{general} therefore yields further examples of domains satisfying P\'olya's conjecture. In particular,  P\'olya's
	conjecture holds for regions constructed by removing finitely many sufficiently small holes from a disk; see Figure~\ref{fig:eccentric-holes}. This may be compared with the annulus case treated in \cite{FLPS26}, where the hole is required to be
	concentric with the original disk.
	
	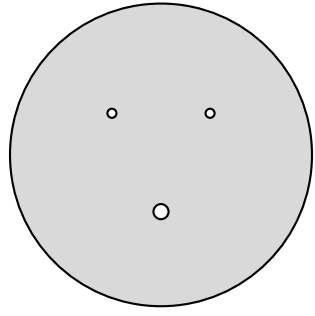
\begin{figure}[htbp]
		\centering
		\begin{tikzpicture}
			\path[fill=gray!30, even odd rule]
			(0,0) circle (2)
			(-0.65,0.55) circle (0.06)
			(0.65,0.55) circle (0.06)
			(0,-0.75) circle (0.10);
			
			\draw[thick] (0,0) circle (2);
			\draw[thick] (-0.65,0.55) circle (0.06);
			\draw[thick] (0.65,0.55) circle (0.06);
			\draw[thick] (0,-0.75) circle (0.10);
		\end{tikzpicture}
		\caption{A disk with three sufficiently small circular holes.}
		\label{fig:eccentric-holes}
	\end{figure}
	
	We also derive explicit constants for the
	unit disk in \S\ref{app}.
	\begin{proposition}\label{diskconstant}
		Let $
		\D:=\{x\in\mathbb R^2:|x|<1\}$. 
		Then, for every $\lambda>0$,
		\begin{equation}
			\label{eq:unit-disk-explicit-remainder-v3}
			\mathcal N_\D^N(\lambda)
			\ge
			\frac{\lambda}{4}
			+
			\max\left(
			\frac{5}{1008}\lambda^{\frac12},
			\frac{269}{7840}
			\right).
		\end{equation}
	\end{proposition}
	
	We also estimate the $L^\infty$-norms of the Neumann eigenfunctions
	of the disk in \S\ref{app}. Combined with Proposition~\ref{diskconstant}, these estimates yield an explicit upper bound on the side length of a square that may be removed from the disk while preserving P\'olya's conjecture, cf. Corollary \ref{cor:coppercoin}. 

\subsection{More domains satisfying the \ref{ENP} inequality}

We also give more domains satisfying the \ref{ENP} inequality without calculating the coefficients explicitly.  
First we present a new class of domains satisfying the \ref{ENP} inequality, which can be viewed as a higher generalization of strip-tiling domains.  Note that this also yields a new class of domains satisfying P\'olya's conjecture in dimensions $n\ge4$.  

\begin{proposition}\label{gentile}
	Let $\Omega\subset\R^n$ be a bounded Lipschitz domain with $n\ge4$. Let
	$U\subset\R^{n_1}$ be a bounded Lipschitz domain,  where
	$2\le n_1\le n-2$, and set $n_2=n-n_1\ge2$.
	Suppose that $\Omega$ tiles $U\times\R^{n_2}$.
	Then $\Omega$ satisfies the \ref{ENP} inequality.
\end{proposition}
 
	For example, given any bounded Lipschitz domain $U\subset\R^{n_1}$, any smooth map $g:U\to\R^{n_2}$, and any invertible matrix $A\in GL(n_2)$, the domain
	\[
	\Omega = \left\{ (x,g(x)+Ay): x\in U,\; y\in(0,1)^{n_2} \right\} 
	\] 
	tiles $U\times\R^{n_2}$ by translations. Indeed,
	\[
	U\times\R^{n_2}
	=
	\bigsqcup_{k\in\mathbb Z^{n_2}}
	\left(\Omega+(0,Ak)\right),
	\]
	where the copies intersect only along their boundaries. Since $g$ is otherwise arbitrary, this construction produces a large and geometrically flexible class of domains that need not have a product structure.

	The main novelty of the proof is a comparison argument based on decompositions of large product domains, which circumvents the lack of domain monotonicity for the Neumann eigenvalue counting function and the resulting failure of the standard tiling argument.

For the remaining of this paper, we verify that classes of domains for which P\'olya's conjecture is known to be true also satisfy the stronger \ref{ENP} inequality. Thus we first list domains for which the Neumann part of P\'olya's conjecture has been established: 
\begin{itemize}
	\item Euclidean tiling domains (\cite{Pol61}  \cite{Kel66}); 
	\item Euclidean balls (\cite{FLPS23}, \cite{FLPS26disk}, \cite{li2026p}); 
	\item product domains with one factor satisfying  the Neumann part of P\'olya's conjecture (\cite{AL97}); 
	\item thin product domains (\cite{HW24}); 
	\item domains that tile a domain satisfying  the Neumann part of P\'olya's conjecture (\cite{FLPS23}).
\end{itemize} 
For further recent progress, see \cite{lin17, Fil25, FL26CPAM, FP19, FLP21, FS23JMP, jingweiimprovement}.
 
As shown above, every Euclidean ball satisfies the \ref{ENP} inequality. We now prove the same for periodic (or regular) Euclidean tiling domains that either have smooth boundary or are polyhedral. General Euclidean tiling domains therefore form the only known class for which the Neumann P\'olya's conjecture holds while the \ref{ENP} inequality remains open.

\begin{definition}[Periodic or regular tiling, \cite{Pol61,Filonov2020}]\label{def:periodic}
	A tiling by isometric copies of $\Omega$ is \emph{periodic} if it is invariant under translations by a full-rank lattice $\Gamma\subset\R^n$ and there are finitely many tile orbits modulo $\Gamma$.  Equivalently, one may select finitely many tile representatives $\Omega_1,\ldots,\Omega_\ell$ such that
	\begin{equation}\label{eq:periodic-supertile}
		C:=\operatorname{int}\!\left(\bigcup_{j=1}^{\ell}\overline{\Omega_j}\right)
	\end{equation}
	is a bounded translational tile for $\Gamma$ modulo null sets.  In particular,
	\begin{equation}\label{eq:period-cell-volume}
		|C|=\ell|\Omega|=\operatorname{covol}(\Gamma).
	\end{equation}
\end{definition}

\begin{proposition}\label{pertil}Let $n\geq2$, and let $\Omega\subset\R^n$ be a bounded periodic Euclidean tiling domain. If $\Omega$ either has smooth boundary or is polyhedral, then it satisfies the \ref{ENP} inequality.
\end{proposition}

P\'olya \cite{Pol61} proved the Neumann P\'olya's conjecture for periodic Euclidean tiling domains. The proof of Proposition \ref{pertil} provides a shorter alternative to P\'olya's argument and strengthens his result to the strict Neumann P\'olya inequality.

\begin{proposition}\label{prop1}
    Let $\Omega=\Omega_1\times\Omega_2$ be a bounded domain in $\R^n$, where $\Omega_1\subset\R^{n_1}$ and $\Omega_2\subset\R^{n_2}$ with $n_1\ge 2$ and $n=n_1+n_2$. Assume that $\Omega_2$ has Lipschitz boundary. If $\Omega_1$ satisfies the Neumann part of P\'olya's conjecture, then $\Omega$ satisfies the \ref{ENP} inequality.
\end{proposition}

\begin{proposition}\label{prop2}
	Let $\Omega_1\subset\R^{n_1}$ and $\Omega_2\subset\R^{n_2}$ be bounded domains with $n_2\ge2$ and $n=n_1+n_2$. Assume that $\Omega_1$ has Lipschitz boundary and $\Omega_2$ has smooth boundary. Then there exists a constant $a_0>0$, depending on $\Omega_1$ and $\Omega_2$, such that for all $a\in(0,a_0)$, $a\Omega_1\times\Omega_2$ satisfies the \ref{ENP} inequality.
\end{proposition}
 
\begin{proposition}\label{tile}
	Let $\Omega,\Omega'\subset\R^n$ be bounded domains. If $\Omega$ satisfies the \ref{ENP} inequality and $\Omega'$ tiles $\Omega$, then $\Omega'$ also satisfies the \ref{ENP} inequality.
\end{proposition}

Since the arguments are similar to those in \cite{AL97,HW24}, 
we provide only a sketch of the proofs for Propositions \ref{prop1} and \ref{prop2}. We also omit the proof of Proposition \ref{tile} since it is standard.

\begin{remark*}
	All domains considered in this subsection also satisfy the EDP inequality \eqref{AssDiri}. The proofs follow the corresponding Dirichlet arguments in the literature and contain no new ingredients; since our focus is on the Neumann case, we omit the details.
\end{remark*}

\subsection*{Plan of the paper}
The paper is organized as follows. In \S\ref{pre}, we provide the necessary preliminaries and preparations for the subsequent proofs. In \S\ref{pfgeneral}, we prove Theorem \ref{general}. In \S\ref{pfthm1.1}, we prove Theorem \ref{mthm} and give explicit calculations for the constants for cuboids. In \S\ref{sec3}, we identify the classes of domains that satisfy the \ref{ENP} inequality. Finally, in Section~\ref{app}, we provide explicit calculations for the disk.

\noindent\textbf{AI-use statement.}
The proof of Proposition \ref{pertil} and those in Section \ref{app} were initially prepared with the assistance of ChatGPT (OpenAI). The authors subsequently verified every step of the proof independently and take full responsibility for the correctness of the final argument.

\section{Inequalities for the counting functions and the Riesz means }\label{pre}

\subsection{Inequalities for general domains}

    Let $\Omega\subset \R^n$ be a bounded Lipschitz domain. The $i$-th nonzero eigenvalue $\mu_i(\Omega)$ of the Neumann Laplacian on $\Omega$ satisfies
    \begin{equation}\label{minmax}
\mu_i(\Omega)
=\min_{\substack{E\subset H^1(\Omega)\\ \dim E=i+1}}
\ \max_{\substack{f\in E\\ f\neq 0}}
\frac{\displaystyle\int_\Omega |\nabla f|^2\,\mathrm d x}
{\displaystyle\int_\Omega f^2\,\mathrm dx}.
\end{equation}

\begin{lemma}\cite[Lemma 2.2]{NS05}\label{cover}
    Let $\Omega \subset \mathbb{R}^n$ be a bounded Lipschitz domain, and suppose that
    $$\bar \Omega=\bar\Omega_1\cup\cdots\cup\bar\Omega_m$$
    where each $\Omega_i$ is open and every $x\in \Omega$ belongs to at most $h$ of the sets $\Omega_i$. Then 
    \[
    \sN^N_\Omega(\lambda)\le \sum_{i=1}^m \sN^N_{\Omega_i}(h\lambda), \quad  \forall  \lambda>0.
    \]
\end{lemma}

\begin{lemma}\cite{See1,See2}\label{Seeasy}
	Let $\Omega\subset \R^n$ be a bounded domain with smooth boundary, then as $\lambda\to +\infty$,
	\[
	\sN^{N}_\Omega(\lambda)=L_{0,n}|\Omega|\lambda^{\frac n2}+O(\lambda^{\frac{n-1}2}).
	\]
\end{lemma}

For a bounded domain $\Omega\subset\R^n$, denote 
\begin{align*}
	I(\Omega):=\inf_{a\in\R^n}\int_\Omega|x-a|^2\,dx,\qquad J_1:= \frac{ |\Omega|^{1/2}}{2|I(\Omega)|^{1/2}}.
\end{align*}

\begin{lemma}\cite[Theorem 1.3]{GJL26}\label{rfofGJL}
	With the notations above,  
	\begin{align}\label{riesz-neumann}
		\sum_{\mu_k(\Omega)<\lambda}\!(\lambda\!-\!\mu_k(\Omega)) \!\ge\! L_{1,n}|\Omega|\lambda^{1+\frac n2} \left(1\!+\!\frac{n(n\!+\!2)}{3} \frac{J_1^4}  {\gamma^{1/2}(4J_1^2 \!+\!\lambda)^{3/2}}
		\right), \quad  \forall  \lambda>0.
	\end{align}
\end{lemma}

By \cite[Theorem 1.1]{FL25}, for any bounded domain $\Omega\subset\mathbb R^d$ ($d\ge 2$) with Lipschitz boundary and any $\gamma> 0$, as $\lambda \to +\infty$,
\begin{equation}\label{FLRiesz}
		\sum_{\mu_k(\Omega)< \lambda} (\lambda-\mu_k(\Omega))^\gamma= L_{\gamma,d}|\Omega| \lambda^{\gamma+\frac{d}{2}}+ \frac{1}{4} L_{\gamma,d-1} |\partial \Omega| \lambda^{\gamma+\frac{d-1}{2}}+ \mathrm o(\lambda^{\gamma+\frac{d-1}{2}}).
\end{equation}
Consequently,  
\begin{lemma}\label{FL3term}
	Let $\Omega\subset\mathbb{R}^d$ ($d\ge 2$) be a bounded domain with Lipschitz boundary. For each $\gamma>0$, there exists a constant $\Lambda(\Omega,\gamma)>0$ such that for all $\lambda > \Lambda(\Omega,\gamma)$,
	\begin{equation*}
		\sum_{\mu_k(\Omega)< \lambda} (\lambda-\mu_k(\Omega))^\gamma\ge L_{\gamma,d}|\Omega| \lambda^{\gamma+\frac{d}{2}}+ \frac{1}{5} L_{\gamma,d-1} |\partial \Omega| \lambda^{\gamma+\frac{d-1}{2}}.
	\end{equation*}
\end{lemma}

\subsection{Precise inequalities for cubes and intervals}

\begin{lemma}\label{upbdIn}
	Let $I_n=(0,\pi)^n$. Then, for every $\lambda>0$,
	\[
	\mathcal{N}_{I_n}^N(\lambda)
	\le
	L_{0,n}\pi^n\lambda^{\frac n2}
	+
	\max\left(2n\lambda^{\frac{n-1}{2}},1\right).
	\]
\end{lemma}

\begin{proof}
	Decomposing a nonnegative lattice point according to its set of nonzero coordinates gives
	\begin{equation}\label{eq:decomposition}
		\mathcal{N}_{I_n}^N(\lambda)
		=
		\sum_{j=0}^n\binom{n}{j}A_j(\sqrt{\lambda}),
	\end{equation}
	where $A_0(r):=1$ and for $j\ge1$, 
	\[
	A_j(r)
	:=	\#\left\{m\in\mathbb Z_{>0}^j:\ |m|<r\right\}
	\le  \frac{\omega_j}{2^j}r^j.
	\]
	Here $\omega_j$ denotes the volume of the
	unit ball in $\mathbb R^j$.  
	
	First suppose  $\lambda \ge 4$. 
	Then $\frac{{\sqrt \lambda}^j}{2^j}\le	2^{1-n}\lambda^{\frac{n-1}2}$ for  $0\le j\le n-1$, and thus 
	\[
	\mathcal{N}_{I_n}^N(\lambda)
	\le 	\sum_{j=0}^{n}
	\binom{n}{j}\frac{\omega_j}{2^j}\sqrt{\lambda}^j
	\le 	\frac{\omega_n}{2^n}\lambda^{\frac n2}+
	2^{1-n}\lambda^{\frac{n-1}2}
	\sum_{j=0}^{n-1}\binom{n}{j}\omega_j.
	\]	 
	So the desired inequality follows from the fact 
	$\frac{\omega_n}{2^n} = L_{0,n}\pi^n$ 
	and 
	\[
	2^{1-n} \sum_{j=0}^{n-1}\binom{n}{j}\omega_j
	\le 2^{1-n} \left( 1+2\sum_{j=1}^{n-1}j\binom{n}{j}
	\right) 
	=\left( 2n+(1-2n)2^{1-n} \right) \le 2n,
	\]
	where we used the trivial identity $\omega_j \le 2j$ and the combinatorial identity 	
	\[ \sum_{j=1}^{n-1}j\binom{n}{j}=\sum_{j=1}^{n-1}n\binom{n-1}{j-1}=n(2^{n-1}-1).\] 
	
	It remains to consider $0<\lambda<4$. Since  
	\[
	\mathcal{N}_{I_n}^N(\lambda)
	=\begin{cases}
		1, & 0<\lambda \le 1,\\
		n+1, & 1 < \lambda \le 2, \\
		\frac{n^2+n+2}{2}, & 2 < \lambda \le 3, \\
		\frac{n^3+5n+6}{6}, & 3 < \lambda \le 4,
	\end{cases}
	\]
	the desired inequality holds for $0<\lambda \le 1$ trivially, and it is not hard to verify for $\lambda \in (k,k+1]$, where $k=1,2,3$, that 
	\[ 
	\mathcal{N}_{I_n}^N(\lambda) \le 	2n k^{\frac{n-1}2} \le 2n\lambda^{\frac{n-1}2}.
	\]
	This completes the proof. 
\end{proof}

The following estimate is related to the one-dimensional Riesz-mean bounds of Harrell and Stubbe \cite[Lemma~3.2]{HarrellStubbe2018}.
In contrast to their estimates, we require the endpoint exponent
$\gamma=\frac12$ and a positive remainder that is uniform for all
$\gamma\ge\frac12$.
\begin{lemma}\label{lwsquare}
	For every $\gamma\ge \frac12$ and $\lambda>0$,
	\[
	\sum_{\substack{i\in\Z_{\ge0}, i^2<\lambda}}
	(\lambda-i^2)^\gamma
	\ge
	 L_{\gamma,1}\pi\lambda^{\gamma+\frac12}
	+\left(1-\frac{\pi}{4}\right)\lambda^\gamma.
	\]
\end{lemma}

\begin{proof}
	We first prove the assertion for $\gamma=\frac12$, in which case $L_{\gamma,1}=\frac 14$.  
	If $0<\lambda\le1$, the desired inequality holds trivially. 
	To prove that the same inequality holds for all $\lambda>1$, we consider, on each interval $(k^2,(k+1)^2]$, $k\ge1$, the function
	\[
	f_k(\lambda)
	:=\sum_{i=0}^{k}\sqrt{\lambda-i^2}
	-\frac{\pi}{4}\lambda-\left(1-\frac{\pi}{4}\right)\sqrt\lambda =\frac{\pi}4 \sqrt \lambda+\sum_{i=1}^{k}\sqrt{\lambda-i^2}
	-\frac{\pi}{4}\lambda.
	\]
	It is concave and thus its minimum is attained at an endpoint. So it
	suffices to verify $f_k(k^2) \ge 0$ and $f_k((k+1)^2) \ge 0$ for all $k$, or equivalently, to prove 
	\begin{equation}\label{eq:square-points}
		\sum_{i=0}^{k}
		\sqrt{k^2-i^2}
		-\frac{\pi}{4}k^2 \ge \left(1-\frac{\pi}{4}\right) k,
		\qquad k\in\mathbb N.
	\end{equation}
	
	Write
	$g_k(x)=\sqrt{k^2-x^2}$ on $[0,k]$.  Since
	$\int_0^k g_k(x)\,dx=\pi k^2/4$, 
	\begin{align*}
		\sum_{i=0}^{k}
		\sqrt{k^2-i^2}
		-\frac{\pi}{4}k^2
		&=\sum_{i=0}^{k-1}
		\left(g_k(i)-\int_i^{i+1}g_k(x)\,dx\right) \\
		&\ge \frac12\sum_{i=1}^{k-1}
		\frac{i}{\sqrt{k^2-i^2}} \\
		&\ge \frac12\int_0^{k-1}
		\frac{x}{\sqrt{k^2-x^2}}\,dx
		=\frac12\left(k-\sqrt{2k-1}\right).
	\end{align*}
    This gives \eqref{eq:square-points} for $k\ge 6$.
	For $k=1,\dots,5$, direct computation shows that the difference
	between the left-hand side and the right-hand side of
	\eqref{eq:square-points} equals 
	\[
	0,\quad
	\sqrt3-\frac\pi2,\quad
	2\sqrt2+\sqrt5-\frac{3\pi}{2},\quad
	\sqrt{15}+2\sqrt3+\sqrt7-3\pi,\quad
	2\sqrt6+\sqrt{21}+7-5\pi
	\]
	respectively, all of which are nonnegative.  This proves the lemma for $\gamma=1/2$. 
	
	Now let $\gamma>\frac12$.  The beta-integral identity
	\[
	(\lambda-i^2)^\gamma
	=
	\mathscr{B}\left(\frac32,\gamma-\frac12\right)^{-1}
	\int_0^\infty
	(\lambda-i^2-t)_+^{\frac 12}
	\,t^{\gamma-\frac32}\,dt,
	\]
	and the already proved endpoint case yield
	\begin{align*}
		 \sum_{i\in \Z_{\ge 0},i^2<\lambda}\!(\lambda-i^2)^\gamma  
		 =&\mathscr{B}\left(\frac32,\gamma-\frac12\right)^{-1}\int_0^\infty \sum_{i\in \Z_{\ge 0},i^2<\lambda}(\lambda-i^2-t)_+^{\frac 12}t^{\gamma-\frac32}\mathrm dt\\
		\ge&\mathscr{B}\left(\frac32,\gamma-\frac12\right)^{-1}\!\int_0^\infty\! \left(\frac\pi4 (\lambda\!-\!t)_++\left(1\!-\!\frac{\pi}{4}\right)(\lambda\!-\!t)_+^{\frac 12}\right)t^{\gamma-\frac32}\mathrm dt\\
		=&L_{\gamma,1}\pi\lambda^{\gamma+\frac12}+\left(1-\frac{\pi}{4}\right)\lambda^\gamma.
	\end{align*}
%
\end{proof}

\section{Proof of Theorem \ref{general}}\label{pfgeneral}

Since translations preserve the Neumann eigenvalues of the Laplacian, we may assume without loss of generality that $x=0$. Accordingly, we write $aU$ in place of $x+aU$ throughout the proof.

First, by Lemma \ref{Seeasy}, there exists a constant $\alpha(U)>0$ such that for all $\lambda>0$,
\[
\sN^N_U(\lambda)\le L_{0,n}|U|\lambda^{\frac n2}+\alpha(U)\lambda^{\frac {n-1}2}+1.
\]
Hence, by Lemma \ref{cover} and the \ref{ENP} inequality, for all $\lambda>0$,
\begin{align*}
	\sN^N_{\Omega'}(\lambda)
	\ge&\sN^N_\Omega(\lambda)-\sN^N_{aU}(\lambda)\\
	=  &\sN^N_\Omega(\lambda)-\sN^N_{U}(a^2\lambda)\\
	\ge&L_{0,n}|\Omega|\lambda^{\frac n2}+C_1(\Omega)\lambda^{\frac{n-1}2}-L_{0,n}|U|(a^2\lambda)^{\frac n2}-\alpha(U)(a^2\lambda)^{\frac {n-1}2}-1\\
	=  &L_{0,n}|\Omega'|\lambda^{\frac n2}+\left(C_1(\Omega)-\alpha(U)a^{n-1}\right)\lambda^{\frac{n-1}2}-1.
\end{align*}
Therefore, if
\[
a<\left(\frac{C_1(\Omega)}{2\alpha(U)}\right)^{\frac 1{n-1}},
\]
then, for
\[
\lambda>\left(\frac 4{C_1(\Omega)}\right)^{\frac 2{n-1}}=:\Lambda_1(\Omega),
\]
we have
\begin{equation}\label{biglambda}
\sN^N_{\Omega'}(\lambda)\ge L_{0,n}|\Omega'|\lambda^{\frac n2}\!+\!\frac {C_1(\Omega)}{2}\lambda^{\frac{n-1}2}-1\ge L_{0,n}|\Omega'|\lambda^{\frac n2}\!+\!\frac {C_1(\Omega)}{4}\lambda^{\frac{n-1}2}.
\end{equation}

It remains to consider the case $\lambda\le\Lambda_1(\Omega)$. Let
\[
N_1=\sN^N_\Omega (\Lambda_1(\Omega)),
\]
and let
$
\{\varphi_0,\varphi_1,\cdots,\varphi_{N_1-1}\}
$
be an orthonormal family of Neumann eigenfunctions on $\Omega$ corresponding to eigenvalues less than $\Lambda_1(\Omega)$. Write 
\[
\Phi:=\max_{0\le i\le N_1-1}\|\varphi_i\|_{L^\infty(\Omega)}.
\]
We prove 
\begin{lemma}\label{genhssta}
	For any $0<\epsilon<1$, if $
	a\le\left(\frac\epsilon{|U|\Phi^2N_1 }\right)^{\frac 1n},$ 
	then  
	\[
	\mu_i(\Omega')\le\frac{\mu_i(\Omega)}{1-\epsilon}, \qquad 0\le i\le N_1-1.
	\]
\end{lemma}

\begin{proof}
	Let $E_i=\operatorname{span}\{\varphi_0,\varphi_1,\ldots,\varphi_i\}.$ 
	By \eqref{minmax},
	\begin{equation}\label{genmax1}
		\mu_i(\Omega')
		\le
		\max_{\substack{f\in E_i\\f\neq0}}
		\frac{\displaystyle\int_{\Omega'}|\nabla f|^2\,dx}
		{\displaystyle\int_{\Omega'}f^2\,dx}.
	\end{equation}
	Write $
	f=\sum_{j=0}^i\alpha_j\varphi_j.$ 
	Then $
	\|f\|_{L^\infty(\Omega)}
	\le
	\Phi\sum_{j=0}^i|\alpha_j|$ and hence
	\begin{align*}
		\frac{\displaystyle\int_{\Omega'}|\nabla f|^2\,dx}
		{\displaystyle\int_{\Omega'}f^2\,dx}
		&\le
		\frac{\displaystyle\int_{\Omega}|\nabla f|^2\,dx}
		{\displaystyle\int_{\Omega}f^2\,dx
			-a^n|U|\Phi^2\cdot\left(\sum_{j=0}^i|\alpha_j|\right)^2}\\
		&=
		\frac{\displaystyle\int_{\Omega}|\nabla f|^2\,dx}
		{\displaystyle\int_{\Omega}f^2\,dx}
		\cdot
		\frac{\displaystyle\int_{\Omega}f^2\,dx}
		{\displaystyle\int_{\Omega}f^2\,dx
			-a^n|U|\Phi^2\cdot\left(\sum_{j=0}^i|\alpha_j|\right)^2}\\
		&\le
		\mu_i(\Omega)
		\cdot
		\frac{\sum_{j=0}^i|\alpha_j|^2}
		{\sum_{j=0}^i|\alpha_j|^2
			-a^n|U|\Phi^2(i+1)\sum_{j=0}^i|\alpha_j|^2}\\
		&\le
		\frac{\mu_i(\Omega)}{1-\epsilon}.
	\end{align*}
	The conclusion follows from \eqref{genmax1}.
\end{proof}

Now let $0<\epsilon<1$ and $0<\lambda\le\Lambda_1(\Omega)$. If
$
\mu_i(\Omega)<(1-\epsilon)\lambda,
$
then necessarily $i\le N_1-1$. Hence, by Lemma \ref{genhssta}, if
\begin{equation}\label{upbdofa}
	a<
	\left(
	\frac{\epsilon}{|U|\Phi^2N_1}
	\right)^{\frac1n},
\end{equation}
then
\[
\mu_i(\Omega')
\le
\frac{\mu_i(\Omega)}{1-\epsilon}
<
\lambda.
\]
Consequently, for all $0<\lambda\le\Lambda_1(\Omega)$,
\[
\sN_{\Omega'}^N(\lambda)
\ge
\sN_\Omega^N((1-\epsilon)\lambda).
\]
Applying the \ref{ENP} inequality gives
\[
\sN_{\Omega'}^N(\lambda)
\ge
L_{0,n}|\Omega|\lambda^{\frac n2}(1-\epsilon)^{\frac n2}
+C_2(\Omega)
\ge
L_{0,n}|\Omega|\lambda^{\frac n2}
\left(1-\frac n2\epsilon\right)
+C_2(\Omega),
\]
where we used $0<\epsilon<1$ and $n\ge 2$. 
Now choose
\[
\epsilon=\min\left(
\frac{C_2(\Omega)}
{L_{0,n}|\Omega|\Lambda_1(\Omega)^{\frac n2}n},\frac 12\right).
\]
Then
\[
C_2(\Omega)
-
L_{0,n}|\Omega|
\Lambda_1(\Omega)^{\frac n2}
\frac n2\epsilon
\ge\frac{C_2(\Omega)}2.
\]
Substituting this choice of $\epsilon$ into \eqref{upbdofa}, we see that if
\[
a<\min\left(
\left(
\frac{C_2(\Omega)}
{L_{0,n}|\Omega|
	\Lambda_1(\Omega)^{\frac n2}
	n|U|\Phi^2N_1}
\right)^{\frac1n},\left(
\frac{1}{2|U|\Phi^2N_1}
\right)^{\frac1n}\right),
\]
then, for all $0<\lambda\le\Lambda_1(\Omega)$,
\[
\begin{aligned}
	\sN_{\Omega'}^N(\lambda)
	&\ge
	L_{0,n}|\Omega|\lambda^{\frac n2}
	+
	C_2(\Omega)
	-
	L_{0,n}|\Omega|
	\Lambda_1(\Omega)^{\frac n2}
	\frac n2\epsilon\\
	&\ge
	L_{0,n}|\Omega|\lambda^{\frac n2}+\frac{C_2(\Omega)}2
	>
	L_{0,n}|\Omega'|\lambda^{\frac n2}+\frac{C_2(\Omega)}2.
\end{aligned}
\]

Combining the estimates for
$0<\lambda\le\Lambda_1(\Omega)$
and
$\lambda>\Lambda_1(\Omega)$,
suppose that
\begin{equation}\label{eq:theexpressionofa}
a\!<\!
\min\!\left(
\!\left(
\frac{C_1(\Omega)}
{2\alpha(U)}
\right)^{\frac1{n\!-\!1}}\!\!,
\,
\left(\!
\frac{C_2(\Omega)}
{L_{0,n}|\Omega|
	\Lambda_1\!(\Omega)^{\frac n2}
	n|U|\Phi^2N_1}\!
\right)^{\frac1n}\!\!,\left(\!
\frac{1}{2|U|\Phi^2N_1}
\!\right)^{\frac1n}\!
\right).
\end{equation}
Define
\[
C_2(\Omega')
:=
\min\left(\frac{C_2(\Omega)}2,1\right),
\qquad
C_1(\Omega')
:=
\frac{C_1(\Omega)}4C_2(\Omega').
\]
Since
\[
C_1(\Omega')
\Lambda_1(\Omega)^{\frac{n-1}{2}}
=C_2(\Omega'),
\]
for $0<\lambda\le\Lambda_1(\Omega)$ we have
\[
\max\left(
C_1(\Omega')\lambda^{\frac{n-1}{2}},
C_2(\Omega')
\right)
=C_2(\Omega')
\le\frac{C_2(\Omega)}2.
\]
On the other hand, for $\lambda>\Lambda_1(\Omega)$,
the inequalities $C_2(\Omega')\le1$ and
$C_1(\Omega')\le C_1(\Omega)/4$ yield
\[
\max\left(
C_1(\Omega')\lambda^{\frac{n-1}{2}},
C_2(\Omega')
\right)
\le
\frac{C_1(\Omega)}4\lambda^{\frac{n-1}{2}}.
\]
Consequently, for all $\lambda>0$,
\[
\sN_{\Omega'}^N(\lambda)
\ge
L_{0,n}|\Omega'|\lambda^{\frac n2}
+
\max\left(
C_1(\Omega')\lambda^{\frac{n-1}{2}},
C_2(\Omega')
\right).
\]

\section{Proof of Theorem \ref{mthm} and Calculations for Cuboids}\label{pfthm1.1}

\subsection{Proof of Theorem \ref{mthm}}

The proof follows \S\ref{pfgeneral}, with only the necessary modifications.

By the \ref{ENP} inequality, Lemmas~\ref{cover} and~\ref{upbdIn}, for all $\lambda>0$,
\begin{align*}
	\sN^N_{\Omega'}(\lambda)
	&\ge
	\sN^N_\Omega(\lambda)
	-
	\sum_{k=1}^M
	\sN^N_{Q_k}(\lambda)\\
	&=\sN^N_\Omega(\lambda)
	-
	\sum_{k=1}^M\sN^N_{I_n}\left(\frac{\ell(Q_k)^2\lambda}{\pi^2}\right)\\
	&\ge
	L_{0,n}|\Omega|
	\lambda^{\frac n2}
	\!+\!
	C_1(\Omega)
	\lambda^{\frac{n-1}2}
	\!-\!
	\sum_{k=1}^M
	\left(
	L_{0,n}\ell(Q_k)^n\lambda^{\frac n2}
	\!+\!
	\frac{2n}{\pi^{n-1}}
	\ell(Q_k)^{n-1}
	\lambda^{\frac{n-1}2}
	\!+\!
	1
	\right)\\
	&=
	L_{0,n}
	|\Omega'|
	\lambda^{\frac n2}
	+
	\left(
	C_1(\Omega)
	-
	\frac{2n}{\pi^{n-1}}
	\mathcal A_{\mathcal Q}
	\right)
	\lambda^{\frac{n-1}2}
	-
	M.
\end{align*}

Hence, if
\begin{equation}\label{1}
\mathcal A_{\mathcal Q}
\le
\frac{C_1(\Omega)\pi^{n-1}}
{4n},
\end{equation}
then for all
\begin{equation}\label{2}
	\lambda>
	\left(
	\frac{2M}
	{C_1(\Omega)}
	\right)^{\frac2{n-1}}:=\Lambda_2(\Omega),
\end{equation}
we have
\[
\sN^N_{\Omega'}(\lambda)
\ge L_{0,n}
|\Omega'|
\lambda^{\frac n2} +\frac{C_1(\Omega)}2\lambda^{\frac{n-1}2}-M
\ge
L_{0,n}
|\Omega'|
\lambda^{\frac n2}.
\]

It remains to consider the case
$\lambda\le\Lambda_2(\Omega)$.
Let
\[
N_2
=
\sN^N_\Omega(\Lambda_2(\Omega)),
\]
and define
\begin{equation}\label{Linftybd}
\Phi
=
\max_{0\le i\le N_2-1}
\|\varphi_i\|_{L^\infty(\Omega)},
\end{equation}
where
$\{\varphi_0,\ldots,\varphi_{N_2-1}\}$
denotes the Neumann eigenfunctions associated with the Neumann eigenvalues below $\Lambda_2(\Omega)$. 
Exactly as in Lemma \ref{genhssta}, one obtains

\begin{lemma}\label{hssta}
	For every $0<\epsilon<1$, if $
	\mathcal V_{\mathcal Q} \le \frac{\epsilon}
	{\Phi^2N_2},$ 
	then  
	\[
	\mu_i(\Omega') \le \frac{\mu_i(\Omega)}
	{1-\epsilon}, \qquad   0\le i\le N_2-1.
	\]
\end{lemma}

Applying the same argument as in \S\ref{pfgeneral}, with the bound on
$|aU|$ replaced by the corresponding bound on $\mathcal V_{\mathcal Q}$, and choosing
\[
\epsilon
=
\min\left(\frac{
	2C_2(\Omega)}
{
	L_{0,n}
	|\Omega|
	\Lambda_2(\Omega)^{\frac n2}
	n},\frac12\right),
\]
we conclude that if
\[
\mathcal V_{\mathcal Q}
<
\min\left(\frac{
	2C_2(\Omega)}
{
	L_{0,n}
	|\Omega|
	\Lambda_2(\Omega)^{\frac n2}
	n\Phi^2N_2},\frac 1{2\Phi^2N_2}\right),
\]
then for all $0<\lambda\le \Lambda_2(\Omega)$,
\[
\sN^N_{\Omega'}(\lambda)
\ge
L_{0,n}
|\Omega'|
\lambda^{\frac n2}.
\]

Combining the estimates for
$\lambda>\Lambda_2(\Omega)$
and
$0<\lambda\le\Lambda_2(\Omega)$
completes the proof.

\subsection{Explicit Calculations for Cuboids}\label{ECfC}
In this subsection, we derive explicit expressions for the constants appearing in Theorem \ref{mthm} in the case of cuboids in $\R^n$, where $n\ge2$.

Let $\sR$ be a cuboid in $\R^n$. Without loss of generality, after scaling and rotation, we may assume that
\[
\sR=(0,a_1)\times\cdots \times (0,a_{n-1})\times (0,\pi),
\]
for some $0<a_1\le\cdots \le a_{n-1}\le \pi$, and let
\[
\widetilde\sR=(0,a_1)\times\cdots \times(0,a_{n-1}).
\]

We first compute the constants appearing in the  \ref{ENP} inequality.

\begin{lemma}\label{lwbdR2}
	For all $\lambda>0$, we have
	\[
	\sN^N_\sR(\lambda)
	\ge
	L_{0,n}|\sR|\lambda^{\frac n2}
	+
	\max\left(
	\left(1-\frac{\pi}{4}\right)
	L_{0,n-1}|\widetilde\sR|
	\lambda^{\frac{n-1}{2}},
	\left(1-\frac{\pi}{4}\right)
	L_{0,n-1}|\widetilde\sR|
	\right).
	\]
\end{lemma}

\begin{proof}
	Since $\widetilde\sR$ satisfies P\'olya's conjecture, for all $\lambda>0$,
	\[
	\begin{aligned}
		\sN^N_{\sR}(\lambda)
		=
		\sum_{i\in\Z_{\ge0},\,i^2<\lambda}
		\sN^N_{\widetilde\sR}(\lambda-i^2)
		\ge
		L_{0,n-1}|\widetilde\sR|
		\sum_{i\in\Z_{\ge0},\,i^2<\lambda}
		(\lambda-i^2)^{\frac{n-1}2}.
	\end{aligned}
	\]
	Applying Lemma~\ref{lwsquare} and \eqref{Cd1d2} gives
	\begin{equation}\label{new}
		\begin{aligned}
			\sN^N_{\sR}(\lambda)
			&\ge
			L_{0,n-1}|\widetilde\sR|
			\left(
			L_{\frac{n-1}{2},1}\pi
			\lambda^{\frac n2}
			+
			\left(1-\frac{\pi}{4}\right)
			\lambda^{\frac{n-1}{2}}
			\right)\\
			&=
			L_{0,n}|\sR|\lambda^{\frac n2}
			+
			\left(1-\frac{\pi}{4}\right)
			L_{0,n-1}|\widetilde\sR|
			\lambda^{\frac{n-1}{2}}.
		\end{aligned}
	\end{equation}
	
	It remains to consider $0<\lambda\leq1$. Since
	$\sN^N_{\sR}(\lambda)=1$, we have for all $0<\lambda\leq1$,
	\[ \sN^N_{\sR}(\lambda) - L_{0,n}|\sR|\lambda^{\frac n2} \ge 1-L_{0,n}|\sR|.
	\]
	For simplicity, write 
	\[
	A_n=L_{0,n}\pi^n.
	\]
	Since $A_{n+2}=A_n\frac{\pi}{2n+4} <A_n,$ 
	we have
	\begin{equation}\label{ineqAn}
		\sup_{n\geq1}A_n
		\leq
		\max(A_1,A_2)=1,
		\qquad
		\sup_{n\geq2}A_n
		\leq
		\max(A_2,A_3)=\frac{\pi}{4}.
	\end{equation}
	Moreover,
	\[
	L_{0,n}|\sR|\leq A_n,
	\qquad
	L_{0,n-1}|\widetilde\sR|\leq A_{n-1}.
	\]
	Hence, for every $n\geq2$,
	\[
		\left(1-\frac{\pi}{4}\right)
		L_{0,n-1}|\widetilde\sR|
		\leq
		\left(1-\frac{\pi}{4}\right)A_{n-1}
		\leq
		1-\frac{\pi}{4}
		\leq
		1-A_n
		\le
		1-L_{0,n}|\sR|.
	\]
	Consequently, for all $0<\lambda\leq1$,
	\[
	\sN^N_{\sR}(\lambda)
	\ge
	L_{0,n}|\sR|\lambda^{\frac n2}
	+
	\left(1-\frac{\pi}{4}\right)
	L_{0,n-1}|\widetilde\sR|.
	\]
	
	For $\lambda\geq1$, the desired result follows from
	\eqref{new}, since $\lambda^{\frac{n-1}{2}}\geq1$. Combining
	this with the preceding estimate completes the proof.
\end{proof}

Let $\{Q_i\}_{i=1}^M$ be a collection of pairwise disjoint open
cubes in $\sR$, and set
\[
\sR'=\sR\setminus\overline{\cup_{i=1}^M Q_i}.
\]
Substituting the constant from Lemma~\ref{lwbdR2} into
\eqref{1} and \eqref{2}, we obtain that if
\begin{equation}\label{dim2ass1}
	\sA_\sQ
	\leq
	\frac{(4-\pi)\pi^{n-1}L_{0,n-1}|\widetilde\sR|}{16n},
	\qquad
	\lambda
	\geq
	\left(
	\frac{8M}
	{(4-\pi)L_{0,n-1}|\widetilde\sR|}
	\right)^{\frac{2}{n-1}}
	=:\Lambda_2,
\end{equation}
then
\[
\sN^N_{\sR'}(\lambda)
\geq
L_{0,n}|\sR'|\lambda^{\frac n2}.
\]

It remains to consider $\lambda<\Lambda_2$. By the explicit
expression for the Neumann eigenvalues of cuboids and
Lemma~\ref{upbdIn},
\[
\begin{aligned}
	\sN^N_\sR(\Lambda_2)
	&\leq
	\sN^N_{I_n}(\Lambda_2)
	\leq
	L_{0,n}\pi^n\Lambda_2^{\frac n2}
	+
	2n\Lambda_2^{\frac{n-1}{2}}
	=:N_2.
\end{aligned}
\]

For any
\[
\ell=(l_1,\ldots,l_n)\in\Z_{\geq0}^n,
\]
let $m(\ell)$ denote the number of nonzero components of $\ell$.
Define $\varphi_\ell:\sR\to\R$ by
\[
\varphi_\ell(x)
=
\frac{2^{\frac{m(\ell)}2}}{\sqrt{|\sR|}}
\left(
\prod_{i=1}^{n-1}
\cos\left(\frac{l_i\pi}{a_i}x_i\right)
\right)
\cos(l_nx_n),
\qquad
x=(x_1,\ldots,x_n)\in\sR.
\]
Then
$
\{\varphi_\ell\}_{\ell\in\Z_{\geq0}^n}
$
is an $L^2(\sR)$-orthonormal basis consisting of Neumann
eigenfunctions on $\sR$. Moreover,
\[
\|\varphi_\ell\|_{L^\infty(\sR)}
\leq
\frac{2^{\frac n2}}{\sqrt{|\sR|}}.
\]

Replacing the constant $\Phi$ in Lemma~\ref{hssta} by $2^{\frac n2}/\sqrt{|\sR|}$,
we immediately obtain the following lemma.

\begin{lemma}\label{eigenR'R}
	Let $0<\epsilon<1$. If $\sV_\sQ \le \frac{|\sR|\epsilon}{2^nN_2},$ 
	then
	\[
	\mu_i(\sR') \leq \frac{\mu_i(\sR)}{1-\epsilon}, \qquad 1\leq i<N_2.
	\]
\end{lemma}

Applying the same argument as in Section~\ref{pfgeneral}, we
conclude that, for $0<\epsilon<1$ and $\lambda\leq\Lambda_2$, if
\begin{equation}\label{upbdVQ}
	\sV_\sQ
	\le
	\frac{|\sR|\epsilon}{2^nN_2},
\end{equation}
then
\begin{equation}\label{R'R}
	\sN_{\sR'}^N(\lambda)
	\geq
	\sN_\sR^N((1-\epsilon)\lambda).
\end{equation}

We first consider the case $0<\lambda\leq1$. By
\eqref{ineqAn},
\[
\sN^N_{\sR'}(\lambda)
\geq
1
\geq
L_{0,n}\pi^n
>
L_{0,n}|\sR'|\lambda^{\frac n2}.
\]
It remains to consider $1<\lambda\leq\Lambda_2$. Combining
Lemma~\ref{lwbdR2} with \eqref{R'R}, we obtain
\[
\begin{aligned}
	\sN_{\sR'}^N(\lambda)
	&\geq
	L_{0,n}|\sR|
	\lambda^{\frac n2}
	(1-\epsilon)^{\frac n2}+
	\left(1-\frac\pi4\right)L_{0,n-1}|\widetilde\sR|
	\lambda^{\frac{n-1}{2}}
	(1-\epsilon)^{\frac{n-1}{2}}.
\end{aligned}
\]

Choose
\[
\epsilon
=
\min\left(
\frac{(4-\pi)L_{0,n-1}}
{4L_{0,n}\sqrt{\Lambda_2}\,n\pi},
\frac1n
\right).
\]
Since $n/2\geq1$, Bernoulli's inequality gives
\[
(1-\epsilon)^{\frac n2}
\geq
1-\frac n2\epsilon,\qquad (1-\epsilon)^{\frac{n-1}{2}}
>
\frac12.
\] 

Substituting this choice of $\epsilon$ into \eqref{upbdVQ}, we
see that if
\begin{equation}\label{finalVQ}
	\sV_\sQ
	\le
	\frac{|\sR|}{2^nN_2}
	\min\left(
	\frac{(4-\pi)L_{0,n-1}}
	{4L_{0,n}\sqrt{\Lambda_2}\,n\pi},
	\frac1n
	\right),
\end{equation}
then, for all $1<\lambda\leq\Lambda_2$,
\[
\begin{aligned}
	\sN_{\sR'}^N(\lambda)
	&\geq
	L_{0,n}|\sR|
	\lambda^{\frac n2}
	\left(1-\frac n2\epsilon\right)+
	\frac{(4-\pi)L_{0,n-1}}{8}
	|\widetilde\sR|
	\lambda^{\frac{n-1}{2}}\\
	&\geq
	L_{0,n}|\sR'|
	\lambda^{\frac n2}
	-
	L_{0,n}|\sR|
	\lambda^{\frac n2}
	\frac n2\epsilon+
	\frac{(4-\pi)L_{0,n-1}}{8}
	|\widetilde\sR|
	\lambda^{\frac{n-1}{2}}.
\end{aligned}
\]
Since
$
|\sR|=\pi|\widetilde\sR|,
$
the definition of $\epsilon$ and the inequality
$\lambda\leq\Lambda_2$ yield
\[
\begin{aligned}
	L_{0,n}|\sR|
	\lambda^{\frac n2}
	\frac n2\epsilon
	&\leq
	\frac{(4-\pi)L_{0,n-1}}{8}
	|\widetilde\sR|
	\frac{\lambda^{\frac n2}}{\sqrt{\Lambda_2}}\leq
	\frac{(4-\pi)L_{0,n-1}}{8}
	|\widetilde\sR|
	\lambda^{\frac{n-1}{2}}.
\end{aligned}
\]
Consequently, for all $1<\lambda\leq\Lambda_2$, 
\[
\sN_{\sR'}^N(\lambda)
\geq
L_{0,n}|\sR'|\lambda^{\frac n2}.
\]

Combining this estimate with \eqref{dim2ass1} and the case
$0<\lambda\leq1$, we obtain the following theorem.

\begin{theorem}\label{thm:cuboid}
	With the notations above, set 
	\[
	\Lambda_2 := \left( \frac{8M}
	{(4-\pi)L_{0,n-1}|\widetilde\sR|}
	\right)^{\frac2{n-1}}  
	\quad \text{ and } \quad 
	N_2 := L_{0,n}\pi^n\Lambda_2^{\frac n2}
	+ 2n\Lambda_2^{\frac{n-1}{2}}.
	\] 	
	Then the Neumann eigenvalues of $\sR' = \sR\setminus \overline{\cup_{i=1}^MQ_i}$
	satisfy P\'olya's conjecture provided
	\[
	\mathcal A_{\mathcal Q}=
	\sum_{k=1}^M\ell(Q_k)^{n-1}
	<
	\frac{(4-\pi)\pi^{n-1}L_{0,n-1}
		|\widetilde\sR|}{16n}
	\]
	and
	\[
	\mathcal V_{\mathcal Q}=\sum_{k=1}^M \ell(Q_k)^n
	<
	\frac{|\sR|}{2^nN_2}
	\min\left(
	\frac{(4-\pi)L_{0,n-1}}
	{4L_{0,n}\sqrt{\Lambda_2}\,n\pi},
	\frac1n
	\right).
	\]
\end{theorem}

In particular, taking
\[
\sR=(0,\pi)\times(0,\pi)
\qquad\text{and}\qquad
M=1
\]
in the preceding argument, we obtain that if $Q\subset\sR$ is an
open square with side length $s\leq 1/60$, then the Neumann eigenvalues of $\sR'=\sR\setminus\overline Q
$ satisfy P\'olya's conjecture.

\section{Domains satisfying the \ref{ENP} inequality}\label{sec3}
\subsection{Proof of Proposition \ref{pertil}}

We first establish the strict Neumann P\'olya inequality for periodic Euclidean tiling domains. We begin with a lattice-counting lemma. The following strict ball version of Blichfeldt's principle follows from the classical averaging argument \cite{Bli14} and the fact that every convex multiple translational tile is a polytope \cite{GRS12}.

\begin{lemma}[Strict Blichfeldt lemma for Euclidean balls]\label{lem:strict-blichfeldt-ball}
	Let $n\ge2$, let $\Lambda\subset\R^n$ be a full-rank lattice with covolume $d_\Lambda$, and let $r>0$.   
	Then there exists a translation $\theta\in\R^n$ such that
	\begin{equation}\label{eq:blichfeldt-ceil}
		\#\bigl((\Lambda+\theta)\cap B_r\bigr)\ge \left[\frac{|B_r|}{d_\Lambda}+1\right].
	\end{equation} 
\end{lemma} 

We next construct Neumann trial functions via shifted dual-lattice plane waves. 
\begin{lemma}\label{lem:periodic-plane-waves}
	Let $C$ be the translational supertile in Definition \ref{def:periodic}, and let
	\[
	\Gamma^*:=\{\xi\in\R^n:\xi\cdot\gamma\in2\pi\mathbb Z
	\text{ for every }\gamma\in\Gamma\}
	\]
	be the dual lattice.  Then for every $\theta\in\R^n$ and every $\lambda>0$,
	\begin{equation}\label{eq:plane-wave-bracket}
		\#\{\xi\in\Gamma^*:|\xi+\theta|^2<\lambda\}
		\le N_C^N(\lambda)
		\le \ell N_\Omega^N(\lambda).
	\end{equation}
\end{lemma}

\begin{proof} 
	For every nonzero $\zeta\in\Gamma^*$, we have 
	\[ 
	\int_C e^{i\zeta\cdot x}\,dx=0.
	\]
	As a result, given any $\theta$, the restrictions to $C$ of
	\[
	u_{\xi,\theta}(x):=e^{i(\xi+\theta)\cdot x},\qquad \xi\in\Gamma^*,
	\]
	are pairwise orthogonal in $L^2(C)$.  They also diagonalize the Dirichlet form on their span:
	\[
	\int_C\nabla u_{\xi,\theta}\cdot
	\overline{\nabla u_{\eta,\theta}}\,dx
	=|\xi+\theta|^2|C|\,\delta_{\xi\eta}.
	\]
	The conclusion follows from the standard min-max principle.  
\end{proof}

\begin{lemma}[Strict P\'olya for periodic tilings]\label{lem:periodic-strict}
	Let $n\ge2$ and let $\Omega\subset\R^n$ be a bounded periodic Euclidean tiling domain with compact Neumann embedding.  Then
	\begin{equation}\label{eq:periodic-strict}
		N^N_\Omega(\lambda)>L_{0,n}|\Omega|\lambda^{\frac n2},\qquad \lambda>0.
	\end{equation}
\end{lemma}

\begin{proof}
	Fix $\lambda>0$ and put $r=\sqrt\lambda$.  For the dual lattice one has $ \operatorname{covol}(\Gamma^*)=\frac{(2\pi)^n}{|C|}.$ 
	Therefore, by \cref{lem:strict-blichfeldt-ball}, some shift $\theta$ satisfies
	\[
	\#\{\xi\in\Gamma^*:|\xi+\theta|^2<\lambda\}\!\ge\! \left[\frac{|B_r||C|}{(2\pi)^n}+1\right]\!>\! \frac{|B_r||C|}{(2\pi)^n}=\frac{\omega_n\lambda^{n/2}|C|}{(2\pi)^n} 
	=\ell L_{0,n}|\Omega|\lambda^{\frac n2}.
	\]
	Combining this with \cref{lem:periodic-plane-waves} yields
	\[
	\ell N_\Omega^N(\lambda)> \ell L_{0,n}|\Omega|\lambda^{\frac n2},
	\]
	and hence the strict P\'olya inequality.
\end{proof}

If $\Omega$ has smooth boundary, Safarov and Filonov \cite{SafarovFilonov2010}, using earlier microlocal results summarized in \cite{SafarovVassiliev1997}, proved that
\begin{equation}\label{eq:SF-one-sided}
	N_\Omega^N(\lambda)
	\ge L_{0,n}|\Omega|\lambda^{n/2}
	+\frac18L_{0,n-1}|\partial\Omega| \lambda^{(n-1)/2}
	+o(\lambda^{(n-1)/2}),\qquad \lambda\to\infty.
\end{equation} 
Combining this estimate with Lemma~\ref{lem:periodic-strict} and the fact that $\sN_\Omega^N(\lambda)\geq1$ for all $\lambda>0$, we conclude that $\Omega$ satisfies the \ref{ENP} inequality.

If $\Omega$ is polyhedral, Vassiliev's classical two-term asymptotic formula for piecewise smooth domains applies, since every bounded polyhedral domain in $\mathbb R^n$, $n\geq2$, satisfies the billiard nonperiodicity hypothesis \cite{Vassiliev1986}. Arguing as above, we conclude that $\Omega$ satisfies the \ref{ENP} inequality.

\subsection{Proof of Proposition \ref{gentile}}\label{pfgentile}

After scaling, one may assume
\[
\diam(\Omega)=1.
\]
Up to rotations and translations, suppose that copies of $\Omega$ tile $U\times \R^{n_2}$. Let $L>0$ be a sufficiently large integer. Set
\[
R_L:=U\times(-L,L)^{n_2}.
\]

Let $N_L$ denote the number of copies of $\Omega$ intersecting $R_L$.
We claim that
\begin{equation}\label{genL}
	|U|(2L)^{n_2}\le N_L|\Omega|\le |U|(2L+2)^{n_2}.
\end{equation}
Indeed, every point of
\[
R_L=U\times(-L,L)^{n_2}
\]
belongs to a unique copy of $\Omega$. Since these copies are pairwise disjoint and each intersects $R_L$, we obtain
\[
|U|(2L)^{n_2}\le N_L|\Omega|.
\]

On the other hand, let $x\in R_L$ and let $\Omega_x$ be the unique copy of $\Omega$ containing $x$. Then for every $y\in\Omega_x$,
\[
\operatorname{dist}(x,y)\le\operatorname{diam}(\Omega)\le1.
\]
Hence
\[
y\in U\times(-L-1,L+1)^{n_2},
\]
which implies
$
\Omega_x\subset R_{L+1}.
$
Therefore every copy of $\Omega$ intersecting $R_L$ is contained in $R_{L+1}$. Since these copies are pairwise disjoint,
\[
N_L|\Omega|
\le
|U|(2L+2)^{n_2},
\]
proving \eqref{genL}.

Let $N_L'$ denote the number of copies of $\Omega$ intersecting $R_{L+1}$ but not intersecting $R_L$. Similarly,
\begin{equation}\label{genL+2}
	N_L'|\Omega|
	\le
	|U|\left((2L+4)^{n_2}-(2L)^{n_2}\right).
\end{equation}
Let $A_L$ be the set obtained by removing from $R_{L+2}$ the closure of the union of the $N_L$ copies of $\Omega$ intersecting $R_L$, and let 
\[
B_L=R_{L+2}\setminus \overline{R_{L+1}}.
\]
Then we easily obtain the following lemma.
\begin{lemma}\label{genlem}
	\begin{enumerate}
		\item The $N_L$ copies of $\Omega$, together with $A_L$, tile $R_{L+2}$ in the sense that the closure of their union equals $\overline{R_{L+2}}$ and they are pairwise disjoint.
		\item The $N_L'$ copies of $\Omega$, together with $B_L$, cover $A_L$. Moreover, every point in $A_L$ is contained in at most two elements of this cover.
		\item The set $B_L$ can be partitioned into $(2L+4)^{n_2}-(2L+2)^{n_2}$ pairwise disjoint copies of $
		\widetilde U:= U\times (0,1)^{n_2}.$
	\end{enumerate}
\end{lemma}

For $\lambda>0$, by Lemma \ref{genlem}, Lemma \ref{cover} and \eqref{genL+2},
\begin{equation*}
	\begin{aligned}
	    \sN^N_{R_{L+2}}(\lambda)
	 \le \sN^N_{A_L}(\lambda)+N_L\cdot \sN^N_\Omega(\lambda)
	 \le \sN^N_{B_L}(2\lambda)+N_L'\cdot \sN^N_\Omega(2\lambda)+N_L\cdot \sN^N_\Omega(\lambda),
	\end{aligned}
\end{equation*}
and thus
\begin{equation}\label{genmon}
	\begin{aligned}
\sN^N_{R_{L+2}}(\lambda)
\le& \left((2L+4)^{n_2}\!-\!(2L+2)^{n_2}\right)\sN^N_{\widetilde U}(2\lambda)\\&+\frac{|U|}{|\Omega|}\left((2L+4)^{n_2}\!-\!(2L)^{n_2}\right)\sN^N_\Omega(2\lambda)  +
N_L\cdot \sN^N_\Omega(\lambda).
\end{aligned}
\end{equation}
Since cubes in $\R^{n_2}$ satisfy P\'olya's conjecture,
\[
\sN^N_{R_{L+2}}(\lambda)=\sum_{i\ge 0}\sN^N_{(-L-2,L+2)^{n_2}}(\lambda-\mu_i(U))\ge \sum_{i\ge 0}L_{0,n_2}(2L)^{n_2}\left(\lambda-\mu_i(U)\right)_+^{\frac {n_2}2}.
\]
Then, by Lemma \ref{FL3term}, there exists $\Lambda>0$  such that for all $\lambda>\Lambda$,
\[
\sum_{i\ge0}\left(\lambda-\mu_i(U)\right)_+^{\frac{n_2}{2}}
\ge
L_{\frac{n_2}{2},n_1}|U|\lambda^{\frac {n_1+n_2}2}
+\frac15L_{\frac{n_2}{2},n_1-1}|\partial U|\lambda^{\frac{n_1+n_2-1}{2}}.
\]
Therefore, by \eqref{Cd1d2} and the identity $n_1+n_2=n$, for all $\lambda>\Lambda$,
\[
\sN^N_{R_{L+2}}(\lambda)\ge L_{0,n}(2L)^{n_2}|U|\lambda^{\frac n2}+\frac 15L_{0,n-1}(2L)^{n_2}|\partial U|\lambda^{\frac {n-1}2}.
\]
Combining this with \eqref{genmon},  we have
\[
\begin{aligned}
	   &\left((2L+4)^{n_2}\!-\!(2L+2)^{n_2}\right)\!\sN^N_{\widetilde U}(2\lambda)\!+\!\frac{|U|}{|\Omega|}\!\left((2L+4)^{n_2}\!-\!(2L)^{n_2}\right)\!\sN^N_\Omega(2\lambda)\!+\! N_L \! \cdot\! \sN^N_\Omega(\lambda)\\
	& \quad \ge L_{0,n}(2L)^{n_2}|U|\lambda^{\frac n2}+\frac 15L_{0,n-1}(2L)^{n_2}|\partial U|\lambda^{\frac {n-1}2}.
\end{aligned}
\]
Dividing both sides by $(2L)^{n_2}|U|$ and letting $L\to\infty$, it follows from \eqref{genL} that for $\lambda>\Lambda$,
\begin{align}\label{prop16-1}
\sN^N_\Omega(\lambda)\ge L_{0,n}|\Omega|\lambda^{\frac n2}+\frac {|\partial U|}{5|U|}L_{0,n-1}|\Omega|\lambda^{\frac{n-1}2}.
\end{align}

On the other hand, since $n_2/2\ge1$, by Lemma \ref{rfofGJL} and the Aizenman--Lieb identity, for all $\lambda>0$,
\[
\sN^N_{R_{L+2}}(\lambda)> L_{0,n}(2L)^{n_2}|U|\lambda^{\frac n2}.
\]
Combining this with \eqref{genmon}, we have
\[
\begin{aligned}
	&\left((2L+4)^{n_2}\!-\!(2L+2)^{n_2}\right)\!\sN^N_{\widetilde U}(2\lambda)\!+\!\frac{|U|}{|\Omega|}\!\left((2L+4)^{n_2}\!-\!(2L)^{n_2}\right)\!\sN^N_\Omega(2\lambda)+N_L \! \cdot\! \sN^N_\Omega(\lambda)\\
& \quad	>L_{0,n}(2L)^{n_2}|U|\lambda^{\frac n2}.
\end{aligned}
\]
Dividing both sides by $(2L)^{n_2}|U|$ and letting $L\to\infty$ again,   we get
\[
\sN^N_\Omega(\lambda)>L_{0,n}|\Omega|\lambda^{\frac n2}, \qquad \forall \lambda>0.
\]
Combining this estimate with the lower bound established above for $\lambda>\Lambda$, together with the fact that $\sN^N_\Omega(\lambda)\ge1$ for all $\lambda>0$, we conclude that there
exists a constant $C_2(\Omega)>0$ such that for all $\lambda>0$,
\[
\sN^N_\Omega(\lambda)\ge L_{0,n}|\Omega|\lambda^{\frac n2}+C_2(\Omega).
\]
Therefore, using \eqref{prop16-1}, and by decreasing the constant $
\frac {|\partial U|}{5|U|}L_{0,n-1}|\Omega|
$ 
if necessary, there exists a constant $C_1(\Omega)>0$ such that for all $\lambda>0$,
\[
\sN_\Omega^N(\lambda)\ge L_{0,n}|\Omega|\lambda^{\frac n2}+\max\left(C_1(\Omega)\lambda^{\frac{n-1}2},C_2(\Omega)\right).
\]
This completes the proof.
\subsection{Proof of Proposition \ref{strip}} 
Suppose that  $\Omega$ tiles $\R^{n-1}\times(0,\pi)$ 
via rotations and translations. For brevity, set
\[
D:=\diam(\Omega).
\]

As in Section~\ref{pfgentile}, let $L$ be a sufficiently large
positive integer, and let $N_L$ denote the number of copies of
$\Omega$ intersecting
\[
R_L:=(-LD,LD)^{n-1}\times(0,\pi).
\]
Then
\begin{equation}\label{NL}
	(2LD)^{n-1}\pi
	\leq
	N_L|\Omega|
	\leq
	(2(L+1)D)^{n-1}\pi.
\end{equation}
Let $N'_L$ denote the number of copies of $\Omega$ intersecting
$R_{L+1}$ but not intersecting $R_L$. Then
\begin{equation}\label{NL'}
	N'_L|\Omega|
	\leq
	\left(
	(2(L+2))^{n-1}-(2L)^{n-1}
	\right)
	D^{n-1}\pi.
\end{equation}
Let $A_L$ be the set obtained by removing from $R_{L+2}$ the
closure of the union of the $N_L$ copies of $\Omega$ intersecting
$R_L$, and define
\[
B_L
:=
R_{L+2}\setminus\overline{R_{L+1}}.
\]

As in Lemma~\ref{genlem}, we have the following lemma.

\begin{lemma}\label{3}
	\begin{enumerate}
		\item
		The $N_L$ copies of $\Omega$, together with $A_L$, tile
		$R_{L+2}$.
		
		\item
		The $N'_L$ copies of $\Omega$, together with $B_L$, cover
		$A_L$. Moreover, every point of $A_L$ belongs to at most
		two elements of this cover.
		
		\item
		The set $B_L$ can be partitioned into $ (2(L+2))^{n-1}-(2(L+1))^{n-1}$
		pairwise disjoint copies of $\widetilde\Omega := (0,D)^{n-1}\times(0,\pi).$
	\end{enumerate}
\end{lemma}

As in Section~\ref{pfgentile}, Lemma~\ref{3},
Lemma~\ref{cover}, and \eqref{NL'} imply that, for every
$\lambda>0$,
\begin{align*}
	\sN^N_{R_{L+2}}(\lambda)
	\leq{}&
	\left(
	(2(L+2))^{n-1}-(2(L+1))^{n-1}
	\right)
	\sN^N_{\widetilde\Omega}(2\lambda)
	\\
	&+
	\frac{
		\left(
		(2(L+2))^{n-1}-(2L)^{n-1}
		\right)
		|\widetilde\Omega|
	}{
		|\Omega|
	}
	\sN^N_\Omega(2\lambda)
	+
	N_L\sN^N_\Omega(\lambda).
\end{align*}

On the other hand, by Lemma~\ref{lwsquare} and the fact that
cubes in $\R^{n-1}$ satisfy P\'olya's conjecture,
\begin{align*}
	\sN^N_{R_{L+2}}(\lambda)
	&=
	\sum_{i\in\Z_{\geq0},i^2<\lambda}
	\sN^N_{(-(L+2)D,(L+2)D)^{n-1}}
	(\lambda-i^2)
	\\
	&\geq
	L_{0,n-1}2^{n-1}(L+2)^{n-1}D^{n-1}
	\sum_{i\in\Z_{\geq0},i^2<\lambda}
	(\lambda-i^2)^{\frac{n-1}{2}}
	\\
	&\geq
	L_{0,n}\!2^{n-1}\!(L \!+\!2)^{n-1}\!
	D^{n-1}\pi\lambda^{\frac n2}
\!+\!
	\left(1\!-\!\frac\pi4\right)L_{0,n-1}\!2^{n-1}\!(L \! +\!2)^{n-1}\!
	D^{n-1}\lambda^{\frac{n-1}{2}}.
\end{align*}
Consequently,
\begin{align*}
	&
	\left(
	(2(L+2))^{n-1}-(2(L+1))^{n-1}
	\right)
	\sN^N_{\widetilde\Omega}(2\lambda)
	\\
	&\quad+
	\frac{
		\left(
		(2(L+2))^{n-1}-(2L)^{n-1}
		\right)
		D^{n-1}\pi
	}{
		|\Omega|
	}
	\sN^N_\Omega(2\lambda)
	+
	N_L\sN^N_\Omega(\lambda)
	\\
	&\geq
	L_{0,n}2^{n-1}(L+2)^{n-1}
	D^{n-1}\pi\lambda^{\frac n2}
	+
	\left(1-\frac\pi4\right)L_{0,n-1}2^{n-1}(L+2)^{n-1}
	D^{n-1}\lambda^{\frac{n-1}{2}}.
\end{align*}

Dividing both sides by
$
(2LD)^{n-1}
$
and letting $L\to\infty$, the first two terms on the left-hand
side vanish. Moreover, it follows from \eqref{NL} that
\[
\lim_{L\to\infty}
\frac{N_L}
{(2LD)^{n-1}}
=
\frac{\pi}{|\Omega|}.
\]
We therefore obtain, for all $\lambda>0$,
\begin{equation}\label{tile-improved-bound}
	\sN^N_\Omega(\lambda)
	\geq
	L_{0,n}|\Omega|\lambda^{\frac n2}
	+
	\frac{(4-\pi)|\Omega|}{4\pi}
	L_{0,n-1}\lambda^{\frac{n-1}{2}}.
\end{equation}

On the other hand,
\[
\sN^N_\Omega(\lambda)\geq1,
\qquad \lambda>0.
\]
Combining this estimate with \eqref{tile-improved-bound} yields
\[
\begin{aligned}
	\sN^N_\Omega(\lambda)
	\geq{}&
	L_{0,n}|\Omega|\lambda^{\frac n2}
	+
	\max\left(
	1-L_{0,n}|\Omega|\lambda^{\frac n2},
	\frac{(4-\pi)|\Omega|}{4\pi}
	L_{0,n-1}\lambda^{\frac{n-1}{2}}
	\right).
\end{aligned}
\]

Set
\[
\lambda_0
:=
\left(
\frac{1}{2L_{0,n}|\Omega|}
\right)^{\frac2n}.
\]
Then
for $0<\lambda\leq\lambda_0$,  
\[
1-L_{0,n}|\Omega|\lambda^{\frac n2}
\geq
1-L_{0,n}|\Omega|\lambda_0^{\frac n2}
= \frac12, 
\]
while for $\lambda\geq\lambda_0$,  
\begin{align*}
	\frac{(4-\pi)|\Omega|}{4\pi}
	L_{0,n-1}\lambda^{\frac{n-1}{2}}
	\!\geq\!
	\frac{(4-\pi)|\Omega|}{4\pi}
	\! L_{0,n-1}\!\lambda_0^{\frac{n-1}{2}}
	\!=\!
	\frac{(4-\pi)|\Omega|^{\frac1n}}{4\pi}
	\! L_{0,n-1}\!
	\left(
	\frac{1}{2L_{0,n}}
	\right)^{\frac{n-1}{n}}.
\end{align*}
Therefore, for every $\lambda>0$,
\begin{align*}
	&
	\max\left(
	1-L_{0,n}|\Omega|\lambda^{\frac n2},
	\frac{\left(4-\pi\right)|\Omega|}{4\pi}
	L_{0,n-1}\lambda^{\frac{n-1}{2}}
	\right)
	\\
	\geq&
	\min\left(
	\frac12,
	\frac{\left(4-\pi\right)|\Omega|^{\frac1n}}{4	\pi}
	L_{0,n-1}
	\left(
	\frac{1}{2L_{0,n}}
	\right)^{\frac{n-1}{n}}
	\right).
\end{align*}

Recall that
\[
L_{0,n}=\frac{\omega_n}{(2\pi)^n}
=\frac{1}{(4\pi)^{n/2}\Gamma\left(\frac n2+1\right)}.
\]
For $n\ge2$, since $\Gamma\left(\frac n2+1\right)\ge1$, we have
$
L_{0,n}\le 1/4\pi.
$
Consequently,
\[
\left(\frac{1}{2L_{0,n}}\right)^{\frac{n-1}{n}}
\ge (2\pi)^{\frac{n-1}{n}}
\ge \sqrt{2\pi}.
\]
Thus, this quantity admits a positive lower bound independent of $n$.
Define
\[
C_1(\Omega)
:=
\frac{\left(4-\pi\right)|\Omega|}{4\pi}L_{0,n-1}
\]
and
\[
C_2(\Omega)
:=
\min\left(
\frac12,
\frac{\left(4-\pi\right)|\Omega|^{\frac1n}}{2\sqrt\pi}
L_{0,n-1}
\right).
\]
It follows that
\[
\sN^N_\Omega(\lambda)
\geq
L_{0,n}|\Omega|\lambda^{\frac n2}
+
\max\left(
C_1(\Omega)\lambda^{\frac{n-1}{2}},
C_2(\Omega)
\right),
\qquad \lambda>0.
\]

\subsection{Proof of Proposition \ref{prop1} and Proposition \ref{prop2}} 

\begin{proof}[Proof of Proposition \ref{prop1}]
Since $\Omega=\Omega_1\times\Omega_2$, the Neumann eigenvalues of $\Omega$ are
\[
\mu_i(\Omega_1)+\mu_j(\Omega_2),\qquad i,j\in\Z_{\ge0}.
\]
Thus, for any $\lambda>0$, under the assumption that $\Omega_1$ satisfies P\'olya's conjecture, 
\[
\sN^N_{\Omega}(\lambda)=\sum_{j:\mu_j(\Omega_2)<\lambda} \sN^N_{\Omega_1}(\lambda-\mu_j(\Omega_2))\ge L_{0,n_1}|\Omega_1|\sum_{j:\mu_j(\Omega_2)<\lambda}(\lambda-\mu_j(\Omega_2))^{\frac{n_1}2}.
\]
Then, by \eqref{FLRiesz}, as $\lambda\to\infty$,
\[
\begin{aligned}
	\sN^N_{\Omega}(\lambda)
	=&L_{0,n_1}|\Omega_1|\left(L_{\frac{n_1}2,n_2}|\Omega_2|\lambda^{\frac{n}{2}}+\frac14L_{\frac{n_1}2,n_2-1}|\partial\Omega_2|\lambda^{\frac{n-1}{2}}+o(\lambda^{\frac{n-1}{2}})\right)\\
	=&L_{0,n}|\Omega|\lambda^{\frac{n}{2}}+\frac 14L_{0,n-1}|\Omega_1| |\partial\Omega_2|\lambda^{\frac{n-1}{2}}+o(\lambda^{\frac{n-1}{2}}).
\end{aligned}
\]
Hence, there exists $\Lambda(\Omega)>0$ such that for all $\lambda>\Lambda(\Omega)$,
\[
\sN^N_{\Omega}(\lambda)\ge L_{0,n}|\Omega|\lambda^{\frac{n}{2}}+\frac 15 L_{0,n-1} |\Omega_1| |\partial\Omega_2|\lambda^{\frac{n-1}{2}}.
\] 
  
On the other hand, since $\frac{n_1}{2}\ge 1$, by Lemma \ref{rfofGJL} and the Aizenman--Lieb identity, for all $\lambda>0$,
\[
\sN^N_{\Omega}(\lambda)>L_{0,n_1}|\Omega_1|L_{\frac{n_1}2,n_2}|\Omega_2|\lambda^{\frac{n}{2}}=L_{0,n}|\Omega|\lambda^{\frac{n}{2}}.
\]
Combining this with the large $\lambda$ asymptotics and the trivial lower bound $\sN^N_\Omega(\lambda)\ge1$, there exist positive constants $C_1(\Omega),C_2(\Omega)$ such that for all $\lambda>0$,
\[
\sN_\Omega^N(\lambda)\ge L_{0,n}|\Omega|\lambda^{\frac n2}+\max\left(C_1(\Omega)\lambda^{\frac{n-1}2},C_2(\Omega)\right).
\]
This verifies the \ref{ENP} inequality.
\end{proof}

\begin{proof}[Proof of Proposition \ref{prop2}]
The Neumann eigenvalues of $a\Omega_1\times\Omega_2$ are
\[
a^{-2}\mu_i(\Omega_1)+\mu_j(\Omega_2),\qquad i,j\in\Z_{\ge0}.
\]
Hence,
\[
\sN^N_{a\Omega_1\times\Omega_2}(\lambda)
=
\sum_{a^{-2}\mu_i(\Omega_1)<\lambda}
\sN^N_{\Omega_2}(\lambda-a^{-2}\mu_i(\Omega_1)).
\]

By Lemma \ref{Seeasy}, there exists a constant $\alpha>0$ such that for all $\lambda>0$,
\[
\sN^N_{\Omega_2}(\lambda)
\ge
L_{0,n_2}|\Omega_2|\lambda^{\frac{n_2}{2}}
-\alpha \lambda^{\frac{n_2-1}{2}}.
\]
Therefore,
\[
\begin{aligned}
	\sN^N_{a\Omega_1\times\Omega_2}(\lambda)
	\ge&
	\,L_{0,n_2}|\Omega_2|a^{-n_2}
	\sum_{\mu_i(\Omega_1)<a^2\lambda}
	(a^2\lambda-\mu_i(\Omega_1))^{\frac{n_2}{2}}
	\\
	&
	-\alpha a^{1-n_2}
	\sum_{\mu_i(\Omega_1)<a^2\lambda}
	(a^2\lambda-\mu_i(\Omega_1))^{\frac{n_2-1}{2}}.
\end{aligned}
\]

Applying Lemma~\ref{FL3term} and \eqref{FLRiesz}, after choosing $a>0$
sufficiently small, there exist $\Lambda>0$ and $c>0$ such that for all $\lambda>a^{-2}\Lambda$,
\[
\sN^N_{a\Omega_1\times\Omega_2}(\lambda)
\ge
L_{0,n}|a\Omega_1\times\Omega_2|
\lambda^{\frac n2}
+c\lambda^{\frac{n-1}2}.
\]

On the other hand, by \cite[Theorem~1.1]{HW24}, after possibly
decreasing $a$,  
\[
\sN^N_{a\Omega_1\times\Omega_2}(\lambda)
>
L_{0,n}|a\Omega_1\times\Omega_2|
\lambda^{\frac n2}, \qquad \forall \lambda>0.
\]
Combining these estimates with the trivial bound $
\sN^N_{a\Omega_1\times\Omega_2}(\lambda)\ge 1$, 
we conclude that there exist positive constants
$C_1(a\Omega_1\times\Omega_2)$ and
$C_2(a\Omega_1\times\Omega_2)$ such that
\[
\sN^N_{a\Omega_1\times\Omega_2}(\lambda)
\ge
L_{0,n}|a\Omega_1\times\Omega_2|
\lambda^{\frac n2}
+\max\!\left(
C_1(a\Omega_1\times\Omega_2)\lambda^{\frac{n-1}2},
C_2(a\Omega_1\times\Omega_2)
\right).
\]
Hence the \ref{ENP} inequality holds.
\end{proof}

\section{Computations with the unit disk}\label{app}

Let 
\begin{equation}\label{eq:BesselFunctionInt}
	J_n(z) = \frac{1}{\pi}\int_0^\pi \cos\bigl(z\sin\theta-n\theta\bigr)\,d\theta .
\end{equation}
denote the Bessel function of the
first kind of order \(n\), and let
\[
0<j'_{n,1}<j'_{n,2}<\cdots
\]
be the positive zeros of \(J_n'\).  It is well known   that 
\[j'_{0,1}>1, \quad j'_{n,1} >n\;(n \ge 1).\]
By separation of variables and the standard Fourier--Bessel theory
for the disk, see, for example, \cite[Section~2.6]{Zelditch2008}, the Neumann spectrum of \(\D\),
as a multiset, consists of 
\[
\lambda_{0,1}=0,
\qquad
\lambda_{0,k+1}=(j'_{0,k})^2,
\qquad
\lambda_{m,k}=(j'_{m,k})^2\quad(m,k\geq1),
\]	 
and each $\lambda_{m,k}$ with $m \ge 1$ has multiplicity two. Moreover, the associated real-valued orthonormal eigenfunctions can be chosen to be  
\[
\varphi_0=\frac{1}{\sqrt{\pi}}, \qquad  
\varphi_{0,k}(r,\theta) = \frac{J_0(j'_{0,k}r)}
{\sqrt{\pi}\,|J_0(j'_{0,k})|} \quad (k\geq1)
\]
and, for each \(m,k\geq1\), 
\begin{align*}
	\varphi_{m,k}^{c}(r,\theta)
	&=
	C_{m,k}J_m(j'_{m,k}r)\cos(m\theta),\\
	\varphi_{m,k}^{s}(r,\theta)
	&=
	C_{m,k}J_m(j'_{m,k}r)\sin(m\theta),
\end{align*}
where
$C_{m,k}$ is the normalizing constant. 

\subsection{Proof of Proposition \ref{diskconstant}}

\begin{proof}
	For simplicity, denote 
	\[
	a:=\frac5{1008},
	\qquad
	c:=\frac{269}{7840},
	\qquad
	x_*:=\frac{c}{a}=\frac{2421}{350}.
	\]
	We prove
	\begin{equation}
		\label{eq:frequency-target-v3}
		\mathcal N_\D^N(x^2) \geq \frac{x^2}{4}+\max(ax,c),
		\qquad x>0.
	\end{equation}
	
	First suppose \(x>x_*\). According to  \cite[(C.22)]{li2026p},
	for \(x\geq20/3\),
	\begin{equation}
		\label{eq:Li-C22-v3}
		\#\{j:\mu_j(\D)\leq x^2\} 	> 	\frac{x^2}{4}
		+x\left(\frac{13}{40}-\frac1\pi\right)
		-\frac{13}{50\pi x}.
	\end{equation} 	
	Using the   lower bound
	\(\pi>3141592/10^6\), we obtain, for $x>x_*$, 
	\begin{align*}
		\frac{13}{40}-\frac1\pi-\frac{13}{50\pi x^2}
		>
		\frac{13}{40}
		-\left(1+\frac{13}{50x_*^2}\right)\frac{10^6}{3141592}
		=
		\frac{456690232967}{92068139178360}
		>
		a.
	\end{align*}     
	Since \(x_* >20/3\), we thus get
	\[
	\#\{j:\mu_j(\D)\leq x^2\} 	>
	\frac{x^2}{4}+ax, \qquad x\geq x_*.
	\]
	As a result, for any $x>x_*$,
	\[
	\mathcal N_\D^N(x^2) \geq \lim_{\varepsilon \to 0+} 	\#\{j:\mu_j(\D)\leq (x-\varepsilon)^2\} 
	\ge  	\frac{x^2}{4}+ax.
	\] 
	
	Next suppose \(0<x\leq x_*\).  As in the Rayleigh--Ritz construction of
	\cite[Section~4.3]{li2026p}, we set 
	\[
	V_{m,\ell}
	:=
	\operatorname{span}\{r^{p_0},\ldots,r^{p_{\ell-1}}\},
	\qquad
	p_i:=m+2i.
	\]
	In this basis, the quadratic-form and Gram matrices are
	\begin{equation}
		\label{eq:disk-ritz-matrices-v3}
		A_{ij}
		=
		\frac{p_ip_j+m^2}{p_i+p_j},
		\qquad
		B_{ij}
		=
		\frac1{p_i+p_j+2},
		\qquad
		0\leq i,j\leq\ell-1,
	\end{equation}
	with \(A_{00}:=0\) when \(m=0\); cf.~\cite[(4.9)]{li2026p}.
	The matrix \(B\) is positive definite.
	
	Let
	\(\theta^{(m,\ell)}_1\leq\cdots\leq
	\theta^{(m,\ell)}_\ell\)
	be the generalized eigenvalues of \(Av=\theta Bv\).  
	The min--max principle gives
	\begin{equation}
		\label{eq:ritz-minmax-v3}
		\lambda_{m,k}\leq\theta^{(m,\ell)}_k,
		\qquad 1\leq k\leq\ell.
	\end{equation}
	For a real symmetric matrix \(M\), let \(n_-(M)\) denote the
	number of its negative eigenvalues, counted with multiplicity.
	For every \(T\), Sylvester's law of inertia gives
	\begin{equation}
		\label{eq:inertia-counts-v3}
		n_-(A-TB)
		=
		\#\bigl\{k:\theta^{(m,\ell)}_k<T\bigr\}.
	\end{equation}
	
	For \(q\in\{1,3,5,6,8\}\), put
	\[
	T_q:=4(q-c).
	\]
	Thus \(T_q/4+c=q\). Exact rational
	\(LDL^{\mathsf T}\) elimination applied to
	\(A_{m,\ell}-T_qB_{m,\ell}\) gives the following pivots. The last
	column is the number of negative pivots and hence, by
	\eqref{eq:inertia-counts-v3}, the number of Ritz eigenvalues below
	\(T_q\).
	\[
	\begin{array}{c|c|c|c|c}
		(m,\ell)&q&T_q&\text{diagonal pivots}&n_-
		\\ \hline
		(1,2)&1&\dfrac{7571}{1960}
		&\left(\dfrac{269}{7840},
		-\dfrac{95442359}{37961280}\right)&1
		\\[2mm]
		(2,2)&3&\dfrac{23251}{1960}
		&\left(\dfrac{269}{11760},
		-\dfrac{875514519}{84358400}\right)&1
		\\[2mm]
		(0,3)&5&\dfrac{38931}{1960}
		&\left(-\dfrac{38931}{3920},
		\dfrac{2703}{15680},
		-\dfrac{698796893}{1907236800}\right)&2
		\\[2mm]
		(3,2)&6&\dfrac{46771}{1960}
		&\left(\dfrac{269}{15680},
		-\dfrac{3249474359}{158172000}\right)&1
		\\[2mm]
		(4,2)&8&\dfrac{62451}{1960}
		&\left(\dfrac{15949}{19600},
		-\dfrac{866322013}{5251686720}\right)&1
		\\[2mm]
		(1,3)&8&\dfrac{62451}{1960}
		&\left(-\dfrac{54611}{7840},
		\dfrac{902174333}{2568901440},
		-\dfrac{11719249600621}{707304677072000}\right)&2
	\end{array}
	\]
	Combining \eqref{eq:ritz-minmax-v3} and
	\eqref{eq:inertia-counts-v3} therefore gives
	\begin{equation}
		\label{eq:finite-spectral-bounds-v3}
		\begin{aligned}
			(j'_{1,1})^2<T_1,
			\quad  
			(j'_{2,1})^2&<T_3,
			&
			(j'_{0,1})^2&<T_5,
			&
			(j'_{3,1})^2&<T_6,
			\\
			\max\bigl\{(j'_{4,1})^2,(j'_{1,2})^2\bigr\}&<T_8.
		\end{aligned}
	\end{equation}
	Indeed, in the \(m=0\) row the two Ritz values below \(T_5\)
	control \(\lambda_{0,2}=(j'_{0,1})^2\), since
	\(\lambda_{0,1}=0\); the last row similarly controls
	\(\lambda_{1,2}=(j'_{1,2})^2\). 
	Taking  multiplicities into account, the bounds in
	\eqref{eq:finite-spectral-bounds-v3} give
	\[
	\begin{array}{c|c}
		\text{range of \(x^2\)}
		&\text{lower bound for \(\mathcal N_\D^N(x^2)\)}
		\\ \hline
		0<x^2\leq T_1&1\\
		T_1<x^2\leq T_3&3\\
		T_3<x^2\leq T_5&5\\
		T_5<x^2\leq T_6&6\\
		T_6<x^2\leq T_8&8\\
		T_8<x^2\leq x_*^2&12.
	\end{array}
	\]
	In each of the first five rows, the displayed lower count is \(q\),
	and hence
	\[
	\frac{x^2}{4}+c
	\leq
	\frac{T_q}{4}+c=q.
	\]
	For the last row, an exact calculation gives
	\[
	\frac{x^2}{4}+c
	\leq
	\frac{x_*^2}{4}+c
	=
	12-\frac{3893}{980000}
	<12.
	\]
	Consequently,
	\begin{equation}
		\label{eq:low-frequency-v3}
		\mathcal N_\D^N(x^2) \geq \frac{x^2}{4}+c,
		\qquad 0<x\leq x_*.
	\end{equation}
	This finishes the proof of  \eqref{eq:frequency-target-v3} and hence  \eqref{eq:unit-disk-explicit-remainder-v3}.
\end{proof}

\subsection{ $L^\infty$-norms of Neumann eigenfunctions}

Next, we compute the $L^\infty$-norms of the set of Neumann  eigenfunctions of the disk mentioned at the beginning of this section. We need two lemmas. 

\begin{lemma}\label{lem:bessel-critical-value}
	Let $a\ge 2n$ satisfies $J_n'(a)=0$, where for $n=0$ we further assume $a>0$, then
	$
	|J_n(a)|> \frac{1}{2\sqrt a}.
	$
\end{lemma}

\begin{proof}
	Set
	\[
	u_n(x)=\sqrt{x}\,J_n(x),
	\qquad
	q_n(x)=1-\frac{n^2-\frac14}{x^2}.
	\]
	Then $u_n$ satisfies $u_n''+q_n(x)u_n=0$.  
	For $x\ge 2n$, we have $q_n(x)>0$ and $q_n'(x)>0$. Hence the Sonin function
	\[
	S_n(x):=u_n(x)^2+\frac{u_n'(x)^2}{q_n(x)}
	\]
	satisfies
	\[
	S_n'(x) = -\frac{q_n'(x)}{q_n(x)^2}u_n'(x)^2 \le 0.
	\]
	Moreover, by the standard asymptotics of $J_n$, 
	\[
	\lim_{x\to\infty}S_n(x)=\frac{2}{\pi}.
	\]
	Therefore $
	S_n(a)>\frac{2}{\pi}.$
	On the other hand,$J_n'(a)=0$ implies  $u_n'(a)=\frac{u_n(a)}{2a}$ and thus
	\[
	S_n(a) = aJ_n(a)^2 \left( 1 + \frac{1}{4\left(a^2-n^2+\frac14\right)} \right) < 2aJ_n(a)^2.
	\]
	So we get $
	2aJ_n(a)^2 >\frac{2}{\pi}$
	and consequently 
	\[
	|J_n(a)|>\frac 1{\sqrt{\pi a}}>\frac{1}{2\sqrt a}.
	\]
\end{proof}

\begin{lemma}\label{lem:bessel-near-turning}
	Let \(n\geq1\), let \(a\) be a positive zero of \(J_n'\), and set
	\[
	M_n(a)=\max_{0\leq x\leq a}|J_n(x)|.
	\]
	Then
	\[
	\int_0^a xJ_n(x)^2\,dx
	\geq \frac{n}{4}M_n(a)^2.
	\]
\end{lemma}

\begin{proof}
	%
	For \(x>n\), put
	\[
	p_n(x)=1-\frac{n^2}{x^2},
	\qquad
	S_n(x)=J_n(x)^2+\frac{J_n'(x)^2}{p_n(x)}.
	\]
	Using the Bessel equation, one obtains for \(x>n\),
	\[
	S_n'(x)
	=
	-\frac{2J_n'(x)^2}{x\,p_n(x)}
	-\frac{p_n'(x)J_n'(x)^2}{p_n(x)^2}
	\leq0.
	\]
	Thus the absolute values of the successive extrema of \(J_n\) are
	nonincreasing.  Since \(j'_{n,1}\) is the first positive critical point and $|J_n|$ is even,
	\[
	M_n(a)=J_n(j'_{n,1})=\sup_{x\in\R}|J_n(x)|.
	\]
	
	By \eqref{eq:BesselFunctionInt}, $J_n$ is entire. Moreover, it is  of exponential type at most $1$ since 
	\[
	|J_n(z)|
	\leq
	\frac{1}{\pi}\int_0^\pi e^{|y||\sin\theta|}\,d\theta
	\leq e^{|y|}
	\leq e^{|z|}.
	\]
	Hence by Bernstein's inequality
	\cite{Bernstein1923}, 
	\[
	\|J_n'\|_{L^\infty(\mathbb R)}
	\leq
	\|J_n\|_{L^\infty(\mathbb R)}
	=M_n(a).
	\]
	As a consequence, for $x\in[j'_{n,1}-1,j'_{n,1}]$,  
	\[
	J_n(x)
	\geq J_n(j'_{n,1})-(j'_{n,1}-x)M_n(a)
	=(1-j'_{n,1}+x)M_n(a).
	\]
	Since \(j'_{n,1}\geq n\geq1\) and \(a\geq j'_{n,1}\), it follows that
	\begin{align*}
		\int_0^a\! xJ_n(x)^2\,dx
		\geq
		M_n(a)^2\int_{j'_{n,1}\!-\!1}^{j'_{n,1}}\! x(1\!-\!j'_{n,1}\!+\!x)^2\,dx=
		\frac{4j'_{n,1}\!-\!1}{12}M_n(a)^2\geq\frac{n}{4}M_n(a)^2.\!
	\end{align*}
\end{proof}

Sharp eigenfunction and spectral-cluster estimates on manifolds with
boundary were established in
\cite{Grieser2002,Sogge2002,SmithSogge2007}. For the unit disk, related
growth estimates were studied in
\cite{LavoiePoliquin2024,Cheng2026}; here we obtain the following
explicit bound.

Consider the special real-valued orthonormal basis of \(L^2(\mathbb D)\),
\[
\mathcal B
=
\{\varphi_0\}
\cup
\{\varphi_{0,k}:k\geq1\}
\cup
\{\varphi_{n,k}^{c},\varphi_{n,k}^{s}:n,k\geq1\}.
\]

\begin{proposition}\label{thm:neumann-disk-supnorm}
	Every \(\varphi\in\mathcal B \setminus \{\varphi_0\}\)   satisfies
	\[
	\|\varphi\|_{L^\infty(\D)}
	\leq 2\,\lambda^{\frac 14}.
	\]
\end{proposition}

\begin{proof}
	By \eqref{eq:BesselFunctionInt},   $|J_n(x)|\leq 1$ for all $n$ and all $x\in\mathbb R$. 
	So if $n=0$, Lemma~\ref{lem:bessel-critical-value} gives  
	\[
	\|\varphi_{0,k}\|_{L^\infty(\mathbb D)}
	\le  \frac{1}{\sqrt\pi\,|J_0(j'_{0,k})|}
	\leq \frac{2}{\sqrt\pi}\sqrt{j'_{0,k}}
	<2\lambda_{0,k+1}^{\frac 14}.
	\]
	Now fix $n,k\geq1$ 
	and write
	\[
	M=M_n(j'_{n,k})=\max_{0\leq x\leq j'_{n,k}}|J_n(x)|.
	\]
	Suppose first that $j'_{n,k}<2n$.  Apply 
	Lemma~\ref{lem:bessel-near-turning} to $a=j'_{n,k}$, 
	\[
	\int_0^1J_n(j'_{n,k}r)^2r\,dr = \frac{1}{{j'_{n,k}}^2}\int_0^{j'_{n,k}} xJ_n(x)^2\,dx \geq \frac{nM^2}{4{j'_{n,k}}^2}.
	\]
	Since the angular factors have \(L^2\)-norm squared \(\pi\) and
	\(L^\infty\)-norm \(1\), we obtain
	\begin{align*}
		\|\varphi_{n,k}^{c,s}\|_{L^\infty(\mathbb D)} 
		\leq
		\frac{M}
		{\left(
			\pi\int_0^1J_n(j'_{n,k}r)^2r\,dr
			\right)^{1/2}} \leq
		\frac{2j'_{n,k}}{\sqrt{\pi n}}
		<
		2\lambda_{n,k}^{\frac14},
	\end{align*}
	where the last inequality follows from $j'_{n,k}<2n<\pi n$.
	
	Finally, if $j'_{n,k}\geq2n$, then 
	$1-n^2/{j'_{n,k}}^2\geq3/4$.  Lemma~\ref{lem:bessel-critical-value} and  $|J_n(x)|\leq 1$ imply
	\begin{align*}
		\|\varphi_{n,k}^{c,s}\|_{L^\infty(\mathbb D)} 
		\le C_{n,k} = 
		\left(\frac{2}{\pi}\right)^{1/2}
		\frac{1}
		{\sqrt{1-n^2/{j'_{n,k}}^2}\,|J_n(j'_{n,k})|} \leq
		4\sqrt{\frac{2}{3\pi}}\sqrt{j'_{n,k}}< 2\lambda_{n,k}^{\frac 14}.
	\end{align*}
	This completes the proof.
\end{proof}
In particular, applying Propositions \ref{diskconstant} and \ref{thm:neumann-disk-supnorm} to the proof of Theorem \ref{mthm} yields  the following corollary.
\begin{corollary}\label{cor:coppercoin}
	For every \(0<a<9\times10^{-8}\), the Neumann eigenvalues of 
	\[
	\mathbb D\setminus\overline{(0,a)\times(0,a)}
	\]
	satisfy P\'olya's conjecture.
\end{corollary}

\bibliographystyle{amsalpha}

\begin{thebibliography}{FLPS26b}
	
	\bibitem[Ber23]{Bernstein1923}
	Serge Bernstein, \emph{Sur une propri{\'e}t{\'e} des fonctions enti{\`e}res},
	Comptes rendus hebdomadaires des s{\'e}ances de l'Acad{\'e}mie des sciences
	\textbf{176} (1923), 1603--1605.
	
	\bibitem[Bli14]{Bli14}Hans F. Blichfeldt, \emph{A new principle in the geometry of numbers, with some applications}, Trans. Amer. Math. Soc. \textbf{15} (1914), no.~3, 227--235.
	
	\bibitem[Che26]{Cheng2026}
	Haoyu Cheng, \emph{Exact {$L^p$} growth rates of Laplace eigenfunctions on the
		unit disk}, arXiv preprint arXiv:2606.02428 (2026).
	
	\bibitem[Fil25]{Fil25}
	Nikolay Filonov, \emph{On the P{\'o}lya Conjecture for the Neumann problem in
		planar convex domains}, Communications on Pure and Applied Mathematics
	\textbf{78} (2025), no.~3, 537--544.
	
	\bibitem[FL25]{FL25}
	Rupert~L Frank and Simon Larson, \emph{Riesz means asymptotics for Dirichlet
		and Neumann Laplacians on Lipschitz domains}, Inventiones mathematicae
	\textbf{241} (2025), no.~3, 999--1079.
	
	\bibitem[FL26]{FL26CPAM}
	\bysame, \emph{Semiclassical inequalities for Dirichlet and Neumann Laplacians
		on convex domains}, Communications on Pure and Applied Mathematics
	\textbf{79} (2026), no.~3, 762--822.
	
	\bibitem[FLO25]{FLO25}
	Veronica Felli, Lorenzo Liverani, and Roberto Ognibene, \emph{Quantitative
		spectral stability for the Neumann Laplacian in domains with small holes},
	Journal of Functional Analysis \textbf{288} (2025), no.~6, 110817.
	
	\bibitem[FLP21]{FLP21}
	Pedro Freitas, Jean Lagac{\'e}, and Jordan Payette, \emph{Optimal unions of
		scaled copies of domains and P{\'o}lya's Conjecture}, Arkiv f{\"o}r Matematik
	\textbf{59} (2021), no.~1, 11--51.
	
	\bibitem[Fil20]{Filonov2020}
	Nikolay Filonov, \emph{On the P\'olya conjecture for the Neumann problem in tiling sets}, arXiv:2006.10663 (2020).
	
	\bibitem[FLPS23]{FLPS23}
	Nikolay Filonov, Michael Levitin, Iosif Polterovich, and David~A Sher,
	\emph{P{\'o}lya’s Conjecture for Euclidean balls}, Inventiones mathematicae
	\textbf{234} (2023), no.~1, 129--169.
	
	\bibitem[FLPS26a]{FLPS26}
	\bysame, \emph{P{\'o}lya's Conjecture for Dirichlet eigenvalues of annuli},
	Journal of the London Mathematical Society \textbf{113} (2026), no.~2,
	e70425.
	
	\bibitem[FLPS26b]{FLPS26disk}
	\bysame, \emph{P\'olya's Conjecture for higher-dimensional Neumann balls},
	arXiv preprint arXiv:2607.29305, 2026.
	
	\bibitem[Fre19]{FP19}
	Pedro Freitas, \emph{A remark on P{\'o}lya’s Conjecture at low frequencies}, Archiv der Mathematik \textbf{112} (2019), no.~3, 305--311.
	
	\bibitem[FS23]{FS23JMP}
	Pedro Freitas and Isabel Salavessa, \emph{Families of non-tiling domains
		satisfying P{\'o}lya's Conjecture}, Journal of Mathematical Physics
	\textbf{64} (2023), no.~12.
	
	\bibitem[GJL26]{GJL26}
	Zaihui Gan, Renjin Jiang, and Fanghua Lin, \emph{The improved Berezin-Li-Yau
		inequality and Kr{\"o}ger inequality and consequences}, Science China Mathematics (2026).
	
	\bibitem[GMWZ]{jingweiimprovement}
	Jingwei Guo, Changxing Miao, Weiwei Wang, and Guoqing Zhan, \emph{Improvement
		of P\'olya's Conjecture for balls and cylinders}, Science China Mathematics 
	(to appear).
	
	\bibitem[GRS12]{GRS12}Nick Gravin, Sinai Robins, and Dmitry Shiryaev, \emph{Translational tilings by a polytope, with multiplicity}, Combinatorica \textbf{32} (2012), no.~6, 629--648.
	
	\bibitem[Gri02]{Grieser2002}
	Daniel Grieser, \emph{Uniform bounds for eigenfunctions of the Laplacian on
		manifolds with boundary}, Communications in Partial Differential Equations
	\textbf{27} (2002), no.~7--8, 1283--1299.
	
	\bibitem[HS18]{HarrellStubbe2018}
	Evans~M Harrell~II and Joachim Stubbe, \emph{Two-term, asymptotically sharp
		estimates for eigenvalue means of the Laplacian}, Journal of Spectral Theory
	\textbf{8} (2018), no.~4, 1529–-1550.
	
	\bibitem[HW]{HW24}
	Xiang He and Zuoqin Wang, \emph{P\'olya's Conjecture for thin products},
	International Mathematics Research Notices (to appear).
	
	\bibitem[Ivr80]{Ivrii80}
	Victor~Ya. Ivrii, \emph{Second term of the spectral asymptotic expansion of the
		Laplace--Beltrami operator on manifolds with boundary}, Functional Analysis
	and Its Applications \textbf{14} (1980), no.~2, 98--106.
	
	\bibitem[JL26]{JL25}
	Renjin Jiang and Fanghua Lin, \emph{P{\'o}lya's Conjecture up to $epsilon$-loss
		and quantitative estimates for the remainder of Weyl's law}, Communications
	on Pure and Applied Mathematics (2026), e70058.
	
	\bibitem[Kab20]{Kab20}
	Yoshitsugu Kabeya, \emph{Eigenvalues of the Laplace-Beltrami operator under the
		homogeneous Neumann condition on a large zonal domain in the unit sphere.},
	Discrete \& Continuous Dynamical Systems: Series A \textbf{40} (2020), no.~6.
	
	\bibitem[Kel66]{Kel66}
	R. Kellner, \emph{On a theorem of P\'olya}. Amer. Math. Monthly, 1966, 73(8): 856--858.
	
	\bibitem[Lap97]{AL97}
	Alekse{\u\i} Laptev, \emph{Dirichlet and Neumann eigenvalue problems on domains
		in Euclidean spaces}, Journal of Functional Analysis \textbf{151} (1997),
	no.~2, 531--545.
	
	\bibitem[LdC12]{LM12}
	Massimo Lanza~de Cristoforis, \emph{Simple Neumann eigenvalues for the Laplace
		operator in a domain with a small hole: A functional analytic approach},
	Revista matem{\'a}tica complutense \textbf{25} (2012), no.~2, 369--412.
	
	\bibitem[Li26]{li2026p}
	Yutian Li, \emph{P\'olya's Conjecture for the Neumann Laplacian on Euclidean
		balls}, arXiv preprint arXiv:2607.25958, 2026.
	
	\bibitem[Lin17]{lin17}
	Fanghua Lin, \emph{Extremum problems of Laplacian eigenvalues and generalized
		P\'olya Conjecture}, Chinese Annals of Mathematics, Series B \textbf{38}
	(2017), no.~2, 497--512.
	
	\bibitem[LNS11]{LNS11}
	Antoine Laurain, Sergei~A Nazarov, and Jan Sokolowski, \emph{Singular
		perturbations of curved boundaries in three dimensions. the spectrum of the
		Neumann Laplacian}, Zeitschrift f{\"u}r Analysis und ihre Anwendungen
	\textbf{30} (2011), no.~2, 145--180.
	
	\bibitem[LP24]{LavoiePoliquin2024}
	Guillaume Lavoie and Guillaume Poliquin, \emph{Growth rates of Laplace
		eigenfunctions on the unit disk}, Annales Mathématiques du Québec
	\textbf{48} (2024), no.~2, 407--425.
	
	\bibitem[Nem17]{Nemes2017}
	Gerg\"o Nemes, \emph{Error bounds for the large-argument asymptotic expansions
		of the hankel and Bessel functions}, Acta Applicandae Mathematicae
	\textbf{150} (2017), no.~1, 141--177.
	
	
	\bibitem[NS05]{NS05}
	Yu~Netrusov and Yu~Safarov, \emph{Weyl asymptotic formula for the Laplacian on
		domains with rough boundaries}, Communications in mathematical physics
	\textbf{253} (2005), no.~2, 481--509.
	
	\bibitem[P{\'o}l54]{Pol54}
	George P{\'o}lya, \emph{Patterns of plausible inference: mathematics and
		plausible reasoning}, Princeton University Press, 1954.
	
	\bibitem[P{\'o}l61]{Pol61}
	\bysame, \emph{On the eigenvalues of vibrating membranes}, Proc. London Math.
	Soc \textbf{11} (1961), no.~3, 419--433.
	
	\bibitem[RT75]{RT75}
	Jeffrey Rauch and Michael Taylor, \emph{Potential and scattering theory on
		wildly perturbed domains}, Journal of Functional Analysis \textbf{18} (1975),
	no.~1, 27--59.

\bibitem[SV97]{SafarovVassiliev97}
	Yuri Safarov and Dmitri Vassiliev, \emph{The asymptotic distribution of
		eigenvalues of partial differential operators}, Translations of Mathematical
	Monographs, vol. 155, American Mathematical Society, Providence, RI, 1997.
	
	\bibitem[SF10]{SafarovFilonov2010}Y.~G. Safarov and N.~D. Filonov, \emph{Asymptotic estimates of the difference between the Dirichlet and Neumann counting functions}, Funct. Anal. Appl. 44 (2010), 286--294.
	
	
	
	\bibitem[See78]{See1}
	R~Seeley, \emph{A sharp asymptotic remainder estimate for the eigenvalues of
		the Laplacian in a domain of $\mathbb R^3$}, Advances in Mathematics \textbf{29} (1978),
	no.~2, 244--269.
	
	\bibitem[See80]{See2}
	\bysame, \emph{An estimate near the boundary for the spectral function of the
		Laplace operator}, American Journal of Mathematics \textbf{102} (1980),
	no.~5, 869--902.
	
	\bibitem[Sog02]{Sogge2002}
	Christopher~D. Sogge, \emph{Eigenfunction and {Bochner--Riesz} estimates on
		manifolds with boundary}, Mathematical Research Letters \textbf{9} (2002),
	no.~2--3, 205--216.
	
	\bibitem[SS07]{SmithSogge2007}
	Hart~F. Smith and Christopher~D. Sogge, \emph{On the {$L^p$} norm of spectral
		clusters for compact manifolds with boundary}, Acta Mathematica \textbf{198}
	(2007), no.~1, 107--153.
	
	
	
	\bibitem[Vas86]{Vassiliev1986}D.~Vassiliev, \emph{Two-term asymptotics of the spectrum of a boundary value problem in the case of a piecewise smooth boundary}, Soviet Math. Dokl. 33 (1986), 227--230; Russian original: Dokl. Akad. Nauk SSSR 286 (1986), 1043--1046.
	
	\bibitem[Zel08]{Zelditch2008}
	Steve Zelditch, \emph{Local and global analysis of eigenfunctions on Riemannian
		manifolds}, Handbook of Geometric Analysis, No. 1, Advanced Lectures in
	Mathematics, vol.~7, International Press, Somerville, MA, 2008, pp.~545--658.
	
\end{thebibliography}
\providecommand{\bysame}{\leavevmode\hbox to3em{\hrulefill}\thinspace}
\providecommand{\MR}{\relax\ifhmode\unskip\space\fi MR }
\providecommand{\MRhref}[2]{%
	\href{http://www.ams.org/mathscinet-getitem?mr=#1}{#2}
}
\providecommand{\href}[2]{#2}

\end{document}